\documentclass{article}

\usepackage{arxiv}
\usepackage{graphicx}
\usepackage[utf8]{inputenc} 
\usepackage[T1]{fontenc}      
\usepackage{amsmath}
\usepackage{amssymb}
\usepackage{amsthm}
\usepackage{graphicx}
\usepackage{color}
\usepackage{hyperref}

\usepackage{lmodern}

\usepackage{mathrsfs}
\usepackage{pgf,tikz,pgfplots}
\usepackage{caption}

\newtheorem{theorem} {Theorem}
\newtheorem{lemma} {Lemma}
\newtheorem{dnt} {Definition}

\newtheorem{claim} {Claim}
\newtheorem{cor} {Corollary}
\newtheorem{prop} {Proposition}
\newtheorem{obs}{Observation}
\newtheorem{con} {Conjecture}
\newtheorem{open problem} {Open Problem}

\numberwithin{lemma}{section}
\numberwithin{theorem}{section}
\numberwithin{cor}{section}
\numberwithin{prop}{section}
\numberwithin{con}{section}
\numberwithin{claim}{section}
\numberwithin{obs}{section}
\numberwithin{dnt}{section}

\usepackage{lineno,hyperref}
\modulolinenumbers[5]
\usepackage[utf8]{inputenc} 
\usepackage[T1]{fontenc}    
\usepackage{hyperref}       
\usepackage{url}            
\usepackage{booktabs}       
\usepackage{amsfonts}       
\usepackage{nicefrac}       
\usepackage{microtype}      
\usepackage{lipsum}
\title{$L(2,1)$-Labeling of the iterated Mycielski of graphs and some related to matching problems}

\author{
	Kamal Dliou \\
	National School of Applied Sciences(ENSA)\\
	Ibn Zohr University\\
	B.P 1136, Agadir, Morocco \\
	\texttt{dlioukamal@gmail.com}
		\And
		Hicham El Boujaoui \\
		National School of Applied Sciences(ENSA)\\
		Ibn Zohr University\\
		B.P 1136, Agadir, Morocco \\
		\texttt{h.elboujaoui@uiz.ac.ma} \\
	
	\And
		Mustapha Kchikech \\
		Modeling and combinatorial laboratory\\
		Polydisciplinary faculty of Safi, Cadi Ayyad University
\\
		Sidi Bouzid, B.P. 4162 - 46000, Safi, Morocco \\
		\texttt{m.kchikech@uca.ac.ma} \\

}
\begin{document}

\maketitle

\begin{abstract}
In this paper, we study the $L(2, 1)$-Labeling of the Mycielski and the iterated Mycielski of graphs in general. For a graph $G$ and all $t\geq 1$, we give sharp bounds for $\lambda(M^t(G))$ the $L(2, 1)$-labeling number of the $t$-th iterated Mycielski in terms of the number of iterations $t$, the order $n$, the maximum degree $\bigtriangleup$, and $\lambda(G)$ the $L(2, 1)$-labeling number of $G$. For $t=1$, we present necessary and sufficient conditions between the $4$-star matching number of the complement graph and $\lambda(M(G))$ the $L(2, 1)$-labeling number of the Mycielski of a graph, with some applications to special graphs. For all $t\geq 2$, we prove that for any graph $G$ of order $n$, we have $2^{t-1}(n+2)-2\leq \lambda(M^t(G))\leq 2^{t}(n+1)-2$. Thereafter, we characterize the graphs achieving the upper bound $2^t(n+1)-2$, then by using the Marriage Theorem and Tutte's characterization of graphs with a perfect $2$-matching, we characterize all graphs without isolated vertices achieving the lower bound $2^{t-1}(n+2)-2$. We determine the $L(2, 1)$-labeling number for the Mycielski and the iterated Mycielski of some graph classes.
\end{abstract}


\section{Introduction}
\par The graphs considered in this paper are finite, simple, and undirected. For graph terminology, we refer to~\cite{dougla}.
\par In 1992, J.R. Griggs and R.K. Yeh \cite{griggs} studied a variation of the frequency assignment problem \cite{hale}, where close transmitters must receive different channels and closer transmitters must receive different channels at least two apart. This problem is known as the $L(2,1)$-Labeling problem, the main target is to come up with a frequency assignment with low-frequency bandwidth.
\par Formally, the $L(2, 1)$-\textit{labeling} of a graph $G=(V,E)$, is a function $f$ from the vertex set $V$ to the set of all nonnegative integers, such that $|f(x)-f(y)|\geq 2$ if $d_G(x, y)=1$ and $|f(x)-f(y)|\geq 1$ if $d_G(x, y)=2$, where $d_G(x, y)$ is the distance between the vertices $x$ and $y$ in $G$. The \textit{span} of an $L(2,1)$-labeling  $f$ is the difference between the largest and the smallest label used by $f$. We may always consider zero as the smallest label used, so that the span is the highest label assigned. A $k$-$L(2,1)$-\textit{labeling} is an $L(2, 1)$-labeling with no label greater than $k$, the minimum $k$ so that $G$ has a $k$-$L(2, 1)$-labeling is called the $L(2,1)$-\textit{labeling number} or $\lambda$-\textit{number} of $G$, and denoted by  $\lambda(G)$. An $L(2, 1)$-labeling with span $\lambda(G)$ is called a $\lambda$-\textit{labeling}. 
\par The $L(2,1)$-labeling has been extensively studied (see surveys \cite{tiziana,yeh}). The determination of the exact value of $\lambda(G)$ is an NP-Hard problem for graphs in general, it is NP-Complete to determine whether a graph admits an $L(2,1)$-labeling with span at most $\lambda \geq 4$ \cite{fiala}, the problem remains NP-Complete even restricted to some graph families (see NP-completeness results references in \cite{tiziana}). Therefore, the aim of the research was to bound the $\lambda$-number for graphs. By using the greedy algorithm, Griggs and Yeh \cite{griggs} proved that  $\lambda(G)\leq \bigtriangleup^2+2\bigtriangleup$ for any graph $G$, where $\bigtriangleup$ is the maximum degree of $G$. This upper bound was later improved by Gon\c{c}alves  in \cite{goncal} to $\bigtriangleup^2+\bigtriangleup-2$, and it is the best known upper bound for $\lambda(G)$ in terms of the maximum degree for graphs in general. Griggs and Yeh \cite{griggs} conjectured that $\lambda(G)\leq \bigtriangleup^2$,   for any graph $G$ with $\bigtriangleup\geq 2$, it is called $\bigtriangleup^2$-conjecture and is one of the most captivating open problems about graph labeling with distance conditions. This conjecture was proven to be true by Havet et al. \cite{havet} for graphs with a large maximum degree. The $L(2,1)$-labeling number attracted attention not only for general graphs but also when considering specific graph classes. The decision version of the $L(2, 1)$-labeling problem has been proven to be polynomial for complete graphs, paths, cycles, wheels, trees, complete $k$-partite graphs, among other few graph classes. For an overview on the subject of the $L(2,1)$-labeling (and its generalizations), we refer the reader to the surveys \cite{tiziana,yeh}.
\par In this paper, we investigate the $L(2, 1)$-labeling of the Mycielski and the iterated Mycielski of graphs. In search of triangle-free graphs with a large chromatic number, Mycielski \cite{mycielski} used the following transformation.
\begin{dnt}\label{def1}
	For a given graph $G=(V, E)$ of order $n$ with $V=\{v_1,v_2,\ldots,v_n\}$. The Mycielski graph of $G$,  denoted $M(G)$, is the graph with vertex set $V\cup V'\cup \{u\}$, where $V'=\{v'_i : v_i \in V\}$ and edge set $E \cup \{v_iv'_j : v_iv_j \in E\} \cup \{v'_iu : v'_i \in V'\}$. The vertex $v'_i$ is called the copy of the vertex $v_i$ and $u$ is called the root of $M (G)$.	
\end{dnt}
The $t$-th iterated Mycielski graph of $G$, denoted $M^t (G)$, is defined recursively with $M^0(G)=G$ and for $t\geq 1$ $M^t(G)=M(M^{t-1}(G))$. If $t=1$, $M^1(G)$ is the Mycielski graph of $G$ and is denoted simply $M(G)$.
It is known that $\chi(M(G))=\chi(G)+1$, and $\omega(M (G))=max(2,\omega(G))$, for any graph $G$, where $\chi(G)$ and $\omega(G)$ are respectively the chromatic number and the clique number of $G$. Many aspects and invariants of the Mycielski graphs have been studied (see for example \cite{boro,cara,changhuang,fisher,larsen,lin,shao}), Mycielski graphs are known to be hard-to-color instances and are used for testing coloring algorithms \cite{cara}. The $L(2,1)$-labeling of the Mycielski of graphs has been previously investigated in \cite{lin} and \cite{shao}. A $4$-\textit{star matching}  $H$ of a graph $G$ is a subgraph such that $H$ is a collection of vertex disjoint star graphs $K_{1,1}$, $K_{1,2}$, $K_{1,3}$ or $K_{1,4}$. The $4$-\textit{star matching number} is the maximum order of a $4$-star matching of $G$. In \cite{lin}, W. Lin and P. Lam gave sufficient conditions on the $4$-star matching number of the complement graph $\overline{G}$, so that $\lambda(M(G))\leq 2n$ and $\lambda(M(G))=2n+k$, for any $k\geq 1$. This allows them to prove that $\lambda(M(G))$ can be computed in polynomial time for graphs with diameter at most $2$, and then give the $\lambda$-number of the Mycielski of complete graph $K_n$, and the Mycielski of the graph join of complete graph and the empty graph. Z. Shao and R. Solis-Oba in \cite{shao}, also studied the $L(2,1)$-labeling number of the Mycielski and the iterated Mycielski of graphs. The authors as well gave the $\lambda$-number of the Mycielski of complete graph, and depending on the number of iterations determine  the exact value or give bounds for $\lambda(M^t(K_n))$, then provided bounds for $\lambda(M^t(G))$ for any graph~$G$. \par
In this paper, we continue the work started by Lin and Lam \cite{lin}, and Shao and Solis-Oba \cite{shao}.
In Section~\ref{sec2}, we give some preliminary results about the Mycielski and iterated Mycielski of graphs, and some previous results on the $L(2,1)$-labeling number of graphs.
\par Section~\ref{sec3} is dedicated to  the $L(2,1)$-labeling number of $M(G)$. First, we provide bounds involving the order $n$, the maximum degree $\bigtriangleup$ and the $\lambda$-number of $G$. Then we complete the equivalence relationship between the $4$-star matching number and the $L(2,1)$-labeling number of the Mycielski of a graph. Afterward, we give applications of this result to the $L(2,1)$-labeling number of the Mycielski of some particular graphs, not mentioned in \cite{lin}. The end of Section~\ref{sec3} is dedicated to graphs with a lower bound  $\lambda(M(G))=n+1$, we give a condition for a graph implying that $\lambda(M(G))=n+1$. Then we determine the $L(2,1)$-labeling number of $M(P_n)$ and $M(C_n)$ the Mycielski graph of path and cycle respectively, which allow us to determine all the connected graphs realizing $\lambda(M(G))$ equal to $4$, $6$ and $7$ respectively. 
\par Section~\ref{sec4} is devoted to the $t$-th iterated Mycielski of graphs with $t\geq 2$. As in Section~\ref{sec3}, we give bounds for $\lambda(M^t(G))$ in terms of the number of iterations $t$, the order,  the maximum degree, and $\lambda(G)$. Then we show that for all $t\geq 2$, $\lambda(M^t(K_n))=|M^t(K_n)|-1=2^t(n+1)-2$, then we characterize all graphs having $\lambda(M^t(G))=|M^t(G)|-1=2^t(n+1)-2$. Later, we give a necessary and sufficient condition for any graph $G$ without isolated vertices achieving a lower bound $2^{t-1}(n+2)-2$ for the $\lambda$-number of the iterated Mycielski of $G$, we apply that to get an upper bound that can be calculated in polynomial time for any graph $G$, then we determine $\lambda(M^t(P_n))$, and $\lambda(M^t(C_n))$. Finally, we propose a weak version of the $\bigtriangleup^2$-conjecture for the $L(2,1)$-labeling of the Mycielski and iterated Mycielski of graphs. 
\section{Preliminaries and previous results}
\label{sec2}
\par For a graph $G$, let $\bigtriangleup_{M^t}$, $deg_{M^t}(x)$, and $d_{M^t}(x,y)$ denote  respectively, the maximum degree, the degree of a vertex $x$, and the distance between the vertices $x$ and $y$ in $M^t(G)$. If $t=1$, we denote simply $\bigtriangleup_M$, $deg_M(x)$, and $d_M(x,y)$. As a consequence of Definition ~\ref{def1}, we have the following.
\begin{lemma}\label{lem2.1}
	If $G$ is a graph of order $n$, then  $|M^t(G)|=2^t(n+1)-1$.
\end{lemma}
\begin{proof} 
	From Definition ~\ref{def1}, we have $|M(G)|=2n+1=2(n+1)-1$. By using induction, we can show that $|M^t(G)|=2^t(n+1)-1$.
\end{proof}
\begin{obs}\label{lem2.2}
	If $H$ is a subgraph of a graph $G$, then for any $t\geq 1$, $M^t(H)$ is a subgraph of $M^t(G)$. 
\end{obs}
\begin{lemma}\label{lem2.3}
	For a graph $G$ of order $n$ and maximum degree $\bigtriangleup$. For any $t\geq1$, we have $\bigtriangleup_{M^t}=max(2^{t-1}(n+1)-1, 2^t\bigtriangleup)$.
\end{lemma}
\begin{proof} 
	By  Definition ~\ref{def1}, we have $deg_{M}(u)=n$, $deg_{M}(x)=2deg_G(x)$, and $deg_{M}(x')=deg_G(x)+1$ for all $x \in V$, where $x'$ is the copy of the vertex $x$ in $M(G)$. Then $\bigtriangleup_{M}=max(n, 2\bigtriangleup)$. 	Suppose that for $k\geq 1$, we have $\bigtriangleup_{M^k}=max(2^{k-1}(n+1)-1, 2^k\bigtriangleup)$. \par
	For $k+1$, if $2^{k-1}(n+1)-1 \geq 2^k\bigtriangleup$, then $\bigtriangleup_{M^k}=2^{k-1}(n+1)-1$. Let $v$ be a vertex of  $M^k(G)$, such that $deg_{M^k}(v)=\bigtriangleup_{M^k}$. From Definition ~\ref{def1} $deg_{M^{k+1}}(v)=2deg_{M^k}(v)=2^k(n+1)-2 \geq deg_{M^{k+1}}(x)$, for all $x \in V_{M^k} \cup V'_{M^k} $. Also  $deg_{M^{k+1}}(u^{k+1})=|M^k(G)|=2^k(n+1)-1 > deg_{M^{k+1}}(v)$, where $u^{k+1}$ is the root of $M^{k+1}(G)$. So $\bigtriangleup_{M^{k+1}}=deg_{M^{k+1}}(u^{k+1})=2^k(n+1)-1$.
	\par  Otherwise, if $2^k\bigtriangleup \geq 2^{k-1}(n+1)$, then by the inductive hypothesis $\bigtriangleup_{M^k}=2^k\bigtriangleup$.  We have $deg_{M^{k+1}}(x)=2deg_{M^k}(x)\leq 2^{k+1} \bigtriangleup$, for all $x \in V_{M^k}$. For $x' \in V'_{M^k}$, $deg_{M^{k+1}}(x')=deg_{M^k}(x)+1 \leq 2^k\bigtriangleup+1\leq 2^{k+1} \bigtriangleup $. Also $ deg_{M^{k+1}}(u^{k+1})=2^k(n+1)-1<2^{k+1} \bigtriangleup$. Thus, $\bigtriangleup_{M^{k+1}}=2^{k+1} \bigtriangleup$. It follows that $\bigtriangleup_{M^{k+1}}=max(2^k(n+1)-1,2^{k+1}\bigtriangleup)$.	
\end{proof}
Notice that $M(G)$ is a connected graph if and only if $G$ has no isolated vertices. The \textit{diameter} of a graph $diam(G)$, is the greatest distance between any pair of vertices in $G$. If $G$ is disconnected, then $diam(G)$ is considered to be infinite. In \cite{fisher}, D.C Fisher et al. proved that $diam(M(G))=min(max(2, diam(G)), 4)$, for every graph $G$ without isolated vertices. The following lemmas are a consequence of the proof of this result and the definition of $M(G)$.
\begin{lemma}\label{lem2.4} \cite{fisher}
	For $v_i$ and $v_j$ two non-isolated vertices in $G$. We have
	$d_{M}(u,v'_i)=1$, $d_{M}(u,v_i)=2$, $d_{M}(v'_i,v'_j)=2$, $d_{M}(v_i,v'_i)=2$, $d_{M}(v_i,v'_j)=min(3,d(v_i,v_j))$, and $d_{M}(v_i,v_j)=min(4,d(v_i,v_j))$.	
\end{lemma}
If $v_i$ is an isolated vertex in $G$, then $v_i$ is isolated in $M(G)$, and $v'_i$ is adjacent to the root $u$.
\begin{lemma} \label{lem2.5}
	If $G$ is a graph without isolated vertices. For $t\geq 1$,	$diam(M^t(G))=min(max(2, diam(G)), 4)$.
\end{lemma}
\begin{proof}
	Based on \cite{fisher}, we have $diam(M(G))=min(max(2, diam(G)),4)$. Suppose that for $k\geq 1$, we have $diam(M^k(G))=min(max(2, diam(G)), 4)$.  We have $M^{k+1}(G)=M(M^k(G))$, so $diam(M^{k+1}(G))= min(max(2, diam(M^k(G)), 4)$. If $diam(G)=1 \text{ or } 2$, then by the inductive hypothesis $diam(M^k(G))=2$, it follows that $diam(M^{k+1}(G))=2$. If $diam(G)=3$, by the inductive hypothesis $diam(M^k(G))=3$ and so $diam(M^{k+1}(G))=3$. By using the same argument if $diam(G)\geq 4$, we get that $diam(M^{k+1}(G))=4$.
\end{proof}
By Lemma~\ref{lem2.5},  if the diameter of a graph $G$ is $1$ or $2$, then the diameter of the $t$-th iterated Mycielski $M^t(G)$ is $2$, for any $t\geq 1$. It is clear from the definition of the $L(2,1)$-Labeling, that any vertices at distance less or equal to $2$ must be assigned  distinct labels. So for any diameter two graph $G$, all the vertices must be assigned  different labels $\lambda(G)\geq |G|-1$. These arguments will also be used throughout the paper.	
\par We recall some previous results on the $L(2,1)$-labeling of graphs.
\begin{lemma} \label{lem2.6}
	\cite{griggs}
	If $G$ is a graph of maximum degree $\bigtriangleup \geq 1$, then $\lambda(G) \geq \bigtriangleup+1$. If  $\lambda(G)=\bigtriangleup +1$, then for every  vertex $v$ of degree $\bigtriangleup $, $f(v)=0$ or $\bigtriangleup +1$ for any $\lambda$-labeling $f$.
\end{lemma}
For $t\geq 1$, from Lemma ~\ref{lem2.6} and Lemma ~\ref{lem2.3}, an obvious lower bound for $\lambda(M^t(G))$ would be $max(2^{t-1}(n+1), 2^t\bigtriangleup+1)$. 
\begin{lemma} \label{lem2.7} \cite{changkuo}
	If $H$ is a subgraph of a graph $G$, then $\lambda(H) \leq \lambda(G)$.
\end{lemma} 
\begin{theorem}\label{th2.1} \cite{griggs}
	If $G$ is a diameter $2$ graph with maximum degree $\bigtriangleup$, then $\lambda(G) \leq \bigtriangleup^2$.
\end{theorem}
In the proof of Theorem ~\ref{th2.1}, Griggs and Yeh proved that for a graph $G$ of order $n$ and maximum degree $\bigtriangleup \geq (n-1)/2\geq 3$, we have $\lambda(G)<\bigtriangleup^2$. Since $\bigtriangleup_{M}=max(n,2\bigtriangleup)$ and $|M(G)|=2n+1$, it means  the $\bigtriangleup^2$-conjecture is true for the Mycielski of any graph $G$ of order $n\geq 3$.
\par The \textit{path covering number} of a graph $p_v(G)$, is the smallest number of vertex-disjoint paths needed to cover all the vertices of a graph $G$. The complement graph $\overline{G}$ of a graph $G$ is the graph whose vertex set is $V$ and where $xy\in E(\overline{G})$ if only if $xy \notin E(G)$. In \cite{george}, Georges et al. related the path covering number of the complement graph $\overline{G} $ to the $L(2,1)$-labeling number of $G$, in the following.
\begin{theorem} \label{th2.2} \cite{george}
	For any graph $G$ of order $n$, we have\\
	$\bullet$ $\lambda(G)\leq n-1$ if and only if $p_v(\overline{G})=1$.\\
	$\bullet$ $\lambda(G)= n+r-2 \text{ if and only if } p_v(\overline{G})=r\geq 2$.	
\end{theorem}
\section{The Mycielski of a graph $M(G)$}\label{sec3}
\subsection{Bounds for the $L(2,1)$-labeling number of $M(G)$}
\begin{theorem}\label{th3.1}
	Let $G$ be a graph of order $n\geq 1$,  and maximum degree $\bigtriangleup\geq 0$, we have 
	$$max(n+1, 2(\bigtriangleup+1)) \leq \lambda(M (G))\leq (n+1)+\lambda(G).$$
\end{theorem}
\begin{proof}
	According to the definition of the Mycielski of a graph, the degree of the root $deg_M(u)=n$, then  $\lambda(M(G))\geq n+1$. Otherwise, for $\bigtriangleup\geq 1$, we have the star graph $K_{1,\bigtriangleup}$ is a subgraph of $G$. Then by Observation ~\ref{lem2.2} and Lemma ~\ref{lem2.7}, we have  $ \lambda(M (G))\geq \lambda(M (K_{1,\bigtriangleup}))$. Since $diam(K_{1,\bigtriangleup})=2$ and $|K_{1,\bigtriangleup}|= \bigtriangleup+1$, it follows that $diam(M(K_{1,\bigtriangleup}))=2$, and $ \lambda(M (K_{1,\bigtriangleup}))\geq |M(K_{1,\bigtriangleup})|-1 =2(\bigtriangleup+1)$. Thus, $ \lambda(M (G))\geq 2(\bigtriangleup+1)$. 
	\par For the upper bound, let $h$ be a  $\lambda$-labeling of $G$. We denote $M (G)$ the Mycielski graph of $G$, with vertex set $V(M(G))=\{v_i,v'_i,u : 1\leq i \leq n\}$, where  $v'_i$ is the copy of $v_i$ in $M (G)$  and $u$ is the root. Since every $\lambda$-labeling must assign the label $0$ to a vertex of $G$, we consider without loss of generality that
	$h(v_n)=0$. We define the following labeling $f$ on $V(M (G))$.	
	$$ f(x)=\begin{cases}
	i-1 \hspace{65pt}if\,\,  x=v'_i,\,\, 1\leq i \leq n,\\
	n+h(v_i) \hspace{45pt}if\,\,  x=v_i,\,\, 1\leq i \leq n, \\
	(n+1)+\lambda(G) \hspace{25pt}if\,\,  x=u.
	\end{cases}  $$	
	Now we will check that $f$ is an $L(2,1)$-labeling of $M (G)$, we get five cases.
	\begin{itemize}
		\item We have $ |f(v'_i)-f(v'_j)|=|i-j|\geq 1$ and $d_{M }(v'_i,v'_j)=2$, for all  $1\leq i,j\leq n$  $ i \neq j$.
		\item  By Lemma ~\ref{lem2.4}, if $d_{M}(v_i,v_j)=1$  (respectively $=2$), then $ d_G(v_i,v_j)=1$ (respectively $=2$). We have $ |f(v_i)-f(v_j)|=|h(v_i)-h(v_j)|$. This means  $|f(v_i)-f(v_j)|\geq 2$,  if  $ d_{M }(v_i,v_j)=1 $ and  $|f(v_i)-f(v_j)|\geq 1$,  if  $d_{M }(v_i,v_j)=2$.
		\item For all $1\leq i,j\leq n$, we have  $|f(v_i)-f(v'_j)|=|n+h(v_i)-j+1|$. The distance two conditions are respected for all the following cases,\\
		$i)$ If $1\leq j\leq n-1$,  then $ |f(v_i)-f(v'_j)|\geq 2$.\\
		$ii)$ If $j=n$ and $i=n$, we have  $ |f(v_n)-f(v'_n)|=1$, and $d_M(v_n,v'_n)\geq 2$.\\
		$iii)$ If $j=n$ and $ d_G(v_i,v_n)=1$, we have $ |h(v_i)-h(v_n)|\geq 2$ , so $h(v_i)\geq 2$. It follows that $|f(v_i)-f(v'_n)|\geq 2$.\\
		$iv)$ If $j=n$ and $ d_G(v_i,v_n)\geq 2$, by Lemma ~\ref{lem2.4} we have  $d_{M }(v_i,v'_n)\geq 2$, and $|f(v_i)-f(v'_n)|\geq 1$.
		\item  For all $1\leq i\leq n $,  $ |f(u)-f(v'_i)|=|(n+1)+\lambda (G)-i+1| \geq 2.$
		\item For all $1\leq i\leq n$,   $ |f(u)-f(v_i)|=|(n+1)+\lambda (G)-(n+h(v_i))|\geq 1$, and $d_{M }(u,v_i)\geq 2$.
	\end{itemize}
	\par 
	So $f$ is an $L(2,1)$-labeling of $M (G)$ with span $(n+1)+\lambda(G)$. Hence $\lambda(M(G) )\leq (n+1)+\lambda(G).$
\end{proof}
\begin{cor} \label{cor3.1}
	If $G$ is a diameter $2$ graph of maximum degree $\bigtriangleup$, then $ \lambda(M(G)) \leq 2(\bigtriangleup^2+1)$.	
\end{cor}
\begin{proof} 
	By Theorem ~\ref{th2.1} for a diameter $2$ graph,  we have  $\lambda(G) \leq \bigtriangleup^2$. Also, we have $|G|=n\leq \bigtriangleup^2+1$, known as the Moore bound due to Hoffman and Singleton \cite{hoffman}. By combining this with the upper bound of Theorem ~\ref{th3.1}, we get that  $ \lambda(M(G)) \leq 2(\bigtriangleup^2+1)$.	
\end{proof}
\par The bound $2(\bigtriangleup^2+1)$ in Corollary ~\ref{cor3.1}, can only be attained by the Mycielski of diameter two Moore graphs \cite{hoffman}, since the diameter of the Mycielski of these graphs is two, and these are the only diameter two graphs with order $\bigtriangleup^2+1$ and  $\lambda$-number equal to $\bigtriangleup^2$ \cite{griggs}. The only known graphs achieving this bound are $C_5$ the cycle of order $5$, the Petersen graph, and the Hoffman-Singleton graph. 
\subsection{$L(2,1)$-labeling number of the Mycielski and the star matching of the complement}
\par By using the upper bound of Theorem~\ref{th3.1} and Theorem~\ref{th2.2}, we can link the $\lambda$-number of $M(G)$ to the path covering of the complement graph $\overline{G}$. So if $p_v(\overline{G})=1$, i.e.  $\overline{G}$ has a Hamiltonian path, then $\lambda (M(G))\leq 2n$, the equality holds for diameter two graphs. Also if  $p_v(\overline{G})\geq 2$, then $\lambda(M(G))\leq 2n+p_v(\overline{G})-1$. But for more relevant conditions, the study of the path covering of the complement of $M(G)$ is required. \par
We can see that for any graph $G$,  $\overline{M}(G)$ the complement of the Mycielski graph of $G$ is a connected graph. The neighborhood of $u$ in $\overline{M}(G) $ is $V$. For all $1 \leq i \leq n$, $v_iv'_i\in E(\overline{M}(G))$. For $i\neq j$, $v'_iv'_j \in E(\overline{M}(G))$. Also  $v_iv'_j,v_iv_j \in E(\overline{M}(G)) $ if and only if $v_iv_j \notin E(G)$.  The subgraph induced by the set $V$ is $\overline{G}$. The subgraph induced by the set $V'$, is the complete graph on $n$ vertices.\par
Let $m$ be an integer greater or equal to $2$. An $m$-\textit{star matching} $H$ of $G$ is a subgraph of $G$, such that each component of $H$ is isomorphic to a star graph $K_{1,i}$, with $1 \leq i\leq m$. The $m$-\textit{star matching number}, denoted $s_m(G)$,  is the maximum order of an $m$-star matching of $G$, an $m$-star matching of order $s_m(G)$, is said to be maximum. If $s_m(G)=|G|$, we say that $G$ has a perfect $m$-star matching, a perfect $m$-star matching is known also as star-factor  or $\{K_{1,1},K_{1,2},\ldots,K_{1,m}\}$-factor \cite{kano,las}, the problem of finding whether or not a graph $G$ admits a perfect $m$-star matching can be solved in polynomial time \cite{kirk}. 
In \cite{lin}, Lin and Lam studied the $m$-star matching and the $m$-star matching number $s_m(G)$. They delivered an algorithm to compute $s_m(G)$ running in $O(|V||E|)$. Then they related the $4$-star matching number of $\overline{G}$ to the path covering number of $\overline{M}(G)$. In the following we denote by $i_4(G)$ the number of vertices unmatched in a maximum $4$-star matching of $G$, i.e. $i_4(G)=n-s_4(G)$. 
\begin{theorem}\label{th3.2} \cite{lin}
	For any graph $G$, we have\\
	$(i)$ if $i_4(\overline{G})\leq 4$, then $p_v(\overline{M}(G))=1$.\\
	$(ii)$ If $i_4(\overline{G})\geq 5$, then $p_v(\overline{M}(G))= \lceil \frac{i_4(\overline{G})}{2} \rceil -1$.
\end{theorem}
We show that the converse holds in both cases, similarly to Theorem ~\ref{th2.2} in~\cite{george}.
\begin{theorem}\label{th3.3}
	For any graph $G$, we have\\
	$(a)$  $i_4(\overline{G})\leq 4$ if and only if $p_v(\overline{M}(G))=1$.\\
	$(b)$  $ \lceil \frac{i_4(\overline{G})}{2} \rceil=r\geq  3$ if and only if $p_v(\overline{M}(G))=r-1$.
\end{theorem}
\begin{proof}
	\par $(a)$ Considering $(i)$  and the contraposition of $(ii)$ in Theorem ~\ref{th3.2}, we get the necessity and sufficiency.
	\par $(b)$ We use induction on $r$. Let $r=3$.
	\begin{claim}
		If $p_v(\overline{M}(G))=2$, then the root $u$ is not an end-vertex of a path in a minimum path covering of $\overline{M}(G)$.   
	\end{claim}
	\begin{proof}
		If $p_v(\overline{M}(G))=2$, let $P^1$ and $P^2$ be the two paths of a minimum path covering of $\overline{M}(G)$, suppose that  $u$ is an end-vertex of $P^1$. Since $u$ is adjacent in $\overline{M}(G)$ to every vertex in $V$, a vertex in $V$ cannot be an end-vertex of $P^2$, otherwise $\overline{M}(G)$ has a Hamiltonian path. So both ends of $P^2$ are from $V'$. Since the subgraph induced by $V'$ is a complete graph, the other extremity of $P^1$ is in $V$. Let $z$ be the other end of $P^1$, $x'$ and $y'$ the ends of $P^2$. Since $u$ is adjacent to $z$, and $x'$ is adjacent to $y'$. If $z'$ the copy of $z$ belongs to $P^1$, we have $z'$ is adjacent to $x'$ and $y'$, we can construct a Hamiltonian path of $\overline{M}(G)$. If $z'$ belongs to $P^2$, since $z$ is adjacent to $z'$, in this case also $\overline{M}(G)$ has a Hamiltonian path, a contradiction.
	\end{proof}
	If $p_v(\overline{M}(G))=2$, let $x,y\in V$, such that $x$ or its copy and $y$ or its copy are end-vertices of the two different paths in a minimum path covering of $\overline{M}(G)$. We consider the graph $H$ with vertex set $V$, and edge set of its complement $E(\overline{H})=E(\overline{G})\cup\{xy\}$. It is clear that $p_v(\overline{M}(H))=1$, and $i_4(\overline{H})\geq i_4(\overline{G})-2$. Since $p_v(\overline{M}(G))=2$, according to $(a)$ we have $4\geq i_4(\overline{H})$, and $i_4(\overline{G})\geq 5$. It follows that $\lceil \frac{i_4(\overline{G})}{2} \rceil=3$. So from Theorem~\ref{th3.2} $(ii)$, we have Theorem~\ref{th3.3} $(b)$ is true for $r=3$.\par
	We suppose that $(b)$ is true for $3\leq r \leq k$, and let $r=k+1$. 
	\par If $p_v(\overline{M}(G))=k$. Let $x,y\in V$, such that $x$ or its copy and $y$ or its copy are end-vertices of two different paths in a minimum path covering of $\overline{M}(G)$.  We consider the graph $H$ with vertex set $V$, and edge set of its complement $E(\overline{H})=E(\overline{G})\cup\{xy\}$. We have $p_v(\overline{M}(H))= k-1$, and $i_4(\overline{H})\geq i_4(\overline{G})-2$. So by the inductive hypothesis $\lceil \frac{i_4(\overline{H})}{2} \rceil=k$, hence $2k+2 \geq i_4(\overline{G})$. Since $p_v(\overline{M}(G))=k$, by the inductive hypothesis $i_4(\overline{G})\geq 2k+1$. It follows that $\lceil \frac{i_4(\overline{G})}{2} \rceil=k+1$. Theorem~\ref{th3.2} $(ii)$ completes the equivalence. 	
\end{proof}
By combining Theorem ~\ref{th2.2} and Theorem ~\ref{th3.3}, we get the following results.
\begin{theorem}\label{th3.4}
	For any graph $G$ of order $n$, we have\\
	$(a)$ $\lambda(M(G))\leq 2n$ if and only if $i_4(\overline{G})\leq 4$.\\
	$(b)$ For any positive integer $r$, we have $$\lambda(M(G))= 2n+r \text{ if and only if }  \lceil \frac{i_4(\overline{G})}{2} \rceil=r+2.$$
\end{theorem}
\par  Next, we give applications of this previous theorem to the $\lambda$-number of the Mycielski of certain graphs.\par  
If the diameter of $G$ is $1$ or $2$, then $diam(M(G))=2$, we can conclude from Theorem~\ref{th3.4} that $\lambda(M(G))=2n+max\{2,\lceil \frac{i_4(\overline{G})}{2} \rceil \}-2$.
\begin{cor}\label{cor3.2}
	Let $G$ be a graph of order $n$, if the clique number $\omega(G)\leq 4 $, then $ \lambda(M (G))\leq 2n$.
\end{cor}

\begin{proof}
	By Theorem ~\ref{th3.4} $(a)$ if $\lambda(M(G))> 2n$, then $i_4(\overline{G})\geq 5$. This means that $\omega(G)\geq 5$. 
\end{proof}
The graphs with clique number less or equal to $4$ in Corollary ~\ref{cor3.2} include trees, planar graphs, and subcubic graphs. \par
If $X$ is any subset of $V$, we denote $N_G(X)$ the set of all vertices in $V$ adjacent to at least one vertex from $X$ in $G$. In \cite{lin}, a criterion for a graph to have a perfect $m$-star matching is given, this appeared also in \cite{kano,kirk,las}.
\begin{theorem}\label{th3.5}\cite{kano,kirk,lin,las}
	A graph $G$ has a perfect $m$-star matching if and only if for any independent set $S$ in $G$, $|N_G(S)|\geq  |S|/m $. 
\end{theorem}
\begin{cor}\label{cor3.3}
	For a graph $G$ of order $n$ and maximum degree $\bigtriangleup \leq n-2 $. If $ 3(n-1)+ \delta \geq 4 \bigtriangleup $, then $ \lambda(M(G))\leq 2n$.
\end{cor}
\begin{proof}
	Let $\overline{\bigtriangleup}$ and $\overline{\delta}$ denote respectively the maximum and minimum degree of the complement graph $\overline{G}$.
	For any independent set $S$ in $\overline{G}$, let $|E_{\overline{G}}(S)|$ denote the number of edges incident to the vertices of $S$ in $\overline{G}$, we have
	\begin{equation}\label{eq1}
|N_{\overline{G}}(S)| \overline{\bigtriangleup} \geq |E_{\overline{G}}(S)| \geq  \overline{\delta} |S| 
	\end{equation}
	 \par If $3(n-1)+ \delta \geq 4 \bigtriangleup $, since $\overline{\bigtriangleup}=(n-1)-\delta$ and $\overline{\delta}=(n-1)-\bigtriangleup$ means  $ 4\overline{\delta}\geq \overline{\bigtriangleup} $, therefore from Inequality (\ref{eq1}) we get that $|N_{\overline{G}}(S)|\geq |S|/4 $, for any $S$  independent set in $\overline{G}$. Then by Theorem ~\ref{th3.5}, $\overline{G}$ has a perfect $4$-star matching. Hence from Theorem ~\ref{th3.4} $(a)$, we have  $ \lambda(M(G))\leq 2n $.
\end{proof}
From Corollary ~\ref{cor3.3}, any regular graph $G$ of order $n$, except complete graphs,  has  $\lambda(M(G))\leq 2n$. In \cite{lin}, it is shown that for complete graph $\lambda(M(K_2))=4$ and  $\lambda(M(K_n))=2n+\lceil \frac{n}{2} \rceil -2$ for $n\geq 3$. Next, we determine the exact $\lambda$-number of the Mycielski of complete $k$-partite graphs.
\begin{cor}
	Let $G$ be a complete $k$-partite graph of order $n$, where the partite sets consist of $p$ sets of order greater or equal $2$  and $q$  singletons.\\
	$\bullet$ If $q\leq 4$, then $\lambda(M(G))=2n$.\\
	$\bullet$ If $q\geq 5$, then $\lambda(M(G))=2n+\lceil \frac{q}{2} \rceil -2.$
\end{cor}
\begin{proof}
	We have $\overline{G}$ is formed of $p$ connected components that are complete graphs of order greater or equal to $2$, and $q$ isolated vertices. Therefore $i_4(\overline{G})=q$. If $q\leq 4$, by Theorem ~\ref{th3.4} $(a)$,  $\lambda(M(G))\leq 2n$. Since $diam(M(G))=2$, it follows that $\lambda(M(G))=2n$. 
	If $q\geq 5$, then by Theorem ~\ref{th3.4} $(b)$, $\lambda(M(G))=2n+\lceil \frac{q}{2} \rceil -2$.~\end{proof}
Let $G_1,G_2$ be two disjoint graphs. The disjoint union of $G_1$ and $G_2$, denoted by $G_1\cup G_2$, is the graph with vertex set $V(G_1)\cup V(G_2)$ and edge set $E(G_1)\cup E(G_2)$. The joint of $G_1$ and $G_2$ denoted  $G_1\vee G_2$ is the graph obtained from $G_1\cup G_2$ by joining each vertex of $G_1$ to each vertex of $G_2$.
\begin{cor}
	Let $G_1,G_2,\ldots ,G_k$ be a collection of disjoint graphs having respectively $n_1,n_2,\ldots, n_k$ vertices. Let $n=\sum_{i=1}^{k} n_i$, then $\lambda(M(G_1\vee G_2\vee \ldots \vee G_k))=2n+max\{2,\lceil \frac{I}{2} \rceil \}-2$, where $I=\sum_{i=1}^{k}i_4(\overline{G_i})$.
\end{cor}
\begin{proof}
	Let $G=G_1\vee G_2\vee \ldots \vee G_k$, we have $\overline{G}=\overline{G_1}\cup\overline{G_2}\cup\ldots \cup \overline{G_k}$. It follows that $i_4(\overline{G})= \sum_{i=1}^{k}i_4(\overline{G_i})=I$. Thus, by Theorem ~\ref{th3.4} $(a)$ if $I\leq 4$, then $\lambda(M(G))\leq 2n$.  Since $diam(G)=2$, it follows that $\lambda(M(G))= 2n$. If $I\geq 5$, from Theorem ~\ref{th3.4} $(b)$, $\lambda(M(G))= 2n+ \lceil \frac{I}{2} \rceil -2 $.
\end{proof}
\subsection{Graphs with $\lambda(M(G))=n+1$}\label{sec3.3}
For $k \geq 1$, the \textit{$k$th power} of a graph $G$ is the graph $G^k$ with vertex set $V$ and edge set $E(G^k)=\{v_iv_j: 1\leq d_G(v_i,v_j)\leq k \}$. Then the square of a graph $G^2$ has the edge set of its complement graph $E(\overline{G^2})=\{v_iv_j:  d_G(v_i,v_j)\geq 3 \}$. Next we give a condition, so that $\lambda (M(G))=n+1$. 
\begin{lemma}\label{lemm3.1}
	In a graph $G$ of order $n$, if the vertex set $V$ can be partitioned into $k\geq 1$ vertex-disjoint cliques in $\overline{G^2}$, such that at least $k-1$ cliques are of order greater or equal $3$. Then $\lambda(M(G))=n+1$.	
\end{lemma}
\begin{proof}
	Let $V=\cup_{r=1}^{k}S_r$, such that  $S_r$ are vertex-disjoint cliques in $\overline{G^2}$ of order $|S_r|=n_r\geq 3$ for $1\leq r\leq k-1$, and $|S_k|=n_k\geq 1$, where $\sum_{r=1}^{k}n_r=n$. For $1\leq r\leq k $, let us denote $S_r=\{v_{i,r}: 1\leq i\leq n_r\}$,  $v'_{i,r}$ is the copy of the vertex $v_{i,r}$, and $u$ is the root of $M(G)$. We have  $d_G(v_{i,r},v_{j,r})\geq 3$ for any two distinct vertices in $S_r$, so a vertex in $S_{r+1}$ can be adjacent to at most one vertex in $S_r$. For $1\leq r\leq k-1$, the cliques $S_r$  in $\overline{G^2}$  are symmetric of order greater or equal $3$, we suppose without loss of generality that $d_G(v_{n_r,r},v_{1,r+1}))\geq 2$, for $1\leq r\leq k-1$. Let $\psi_1=0$ and for $r\geq 2$, $\psi_r=\sum_{j=1}^{r-1}n_j$.  With respect to the previous assumption, we label the vertices of $M(G)$ as following.  
	\begin{itemize}
		\item For $1\leq r\leq k-1$, define   $f(v_{1,r})=\psi_r$. For $2\leq i\leq n_r$, $f(v_{i,r})=\psi_r+1$. Also $f(v'_{1,r})=\psi_r+1$, and $f(v'_{2,r})=\psi_r$.  For $3\leq i\leq n_r$, $f(v'_{i,r})=\psi_r+i-1$.
		\item If $|S_k|=1$, let $f(v_{1,k})=n$, and $f(v'_{1,k})=n-1$.
		\item If  $|S_k|=2$, let $f(v_{1,k})=n-2$, $f(v'_{1,k})=n-1$, $f(v_{2,k})=n-1$, and $f(v'_{2,k})=n-2$.
		\item If $|S_k|\geq 3$, define $f(v_{1,k})=\psi_k$. For $2\leq i\leq n_r$, $f(v_{i,k})=\psi_k+1$. Also $f(v'_{1,k})=\psi_k+1$, and $f(v'_{2,k})=\psi_k$. For $3\leq i\leq n_k$, $f(v'_{i,k})=\psi_k+i-1$.
	\end{itemize} 
	\par Finally, label the root $f(u)=n+1$. We have $d_G(v_{i,r},v_{j,r}))\geq 3$, and for $1\leq r\leq k-1$ we have $d_G(v_{n_r,r},v_{1,r+1}))\geq 2$. This means by Lemma~\ref{lem2.4} that $d_M(v_{i,r},v_{j,r}))\geq 3$, $d_M(v'_{i,r},v_{j,r}))= 3$, and  $d_M(v'_{n_r,r},v_{1,r+1}))\geq 2$. The labeling $f$ is an $L(2,1)$-labeling of $M(G)$ with span $n+1$.
	Hence $\lambda(M(G))=n+1$.
\end{proof}
\par In the case of the empty graph $\overline{K_n}$, we have $M(\overline{K_n})\cong K_{1,n}\cup \overline{K_n} $. Since $\lambda(K_{1,n})=n+1$, we have $\lambda (M(\overline{K_n}))=n+1$, we can get the same result using Lemma~\ref{lemm3.1}. We are now interested in some connected graphs, we consider the graph path $P_n$ and cycle $C_n$.
\par Let $P_n$ denote the graph path of order $n\geq 3$, with vertex set $V(P_n)=\{v_1,\ldots,v_n\}$ and edge set $E(P_n)=\{v_iv_{i+1}:1\leq i\leq n-1\}$. Denote $V(M( P_n))= V(P_n)\cup  \{v'_i : 1\leq i \leq n\} \cup \{u\}$, where $v'_i$ is the copy of the vertex $v_i$, and $u$ is the root of $M( P_n)$. 
\begin{prop}\label{th3.6}
	$$\lambda(M(P_n))=\begin{cases} 
	6 \hspace{35pt}if\,\, n=3,4, \\
	7 \hspace{35pt}if\,\, n=5, \\
	n+1 \hspace{15pt}if\,\,n \geq 6.
	\end{cases} 	$$
\end{prop}

\begin{proof} 
	\leavevmode
	\begin{itemize}
		\item  For $n=3$, we have $diam(P_3)=2$.   So from Theorem~\ref{th3.4} $\lambda(M(P_3))=6$.
		\item   For $n$=4, we have a $6$-$L(2,1)$-labeling of $M(P_4)$ shown in Figure~\ref{fig1}. Hence $\lambda(M(P_4)) \leq~6$.
		Also we have  $M(P_3)$ is a subgraph of $ M(P_4)$. By Lemma~\ref{lem2.7}, it follows that $\lambda(M(P_4)) \geq \lambda(M(P_3))=6$. thus,  $\lambda(M(P_4))=6$.	
		\item For $n=5$, Figure~\ref{fig2} illustrates a $7$-$L(2,1)$-labeling of $M(P_5)$. This implies also by Theorem~\ref{th3.1}, that $ 6 \leq \lambda(M(P_5)) \leq 7$.
		\begin{figure}[h]
			\centering
			\begin{minipage}  {.5\textwidth} 
				\centering
				\begin{tikzpicture} [scale=0.65]
				\clip(-4.86,-2.4) rectangle (0.9,0.52);
				\draw [line width=0.7pt] (-2,0)-- (-3,-1);
				\draw [line width=0.7pt] (-2,0)-- (-4.48,-1);
				\draw [line width=0.7pt] (-2,0)-- (-1,-1);
				\draw [line width=0.7pt] (-2,0)-- (0.5,-1);
				\draw [line width=0.7pt] (0.5,-1)-- (-1,-2.21);
				\draw [line width=0.7pt] (-1,-1)-- (0.5,-2.21);
				\draw [line width=0.7pt] (-1,-1)-- (-3,-2.21);
				\draw [line width=0.7pt] (-1,-2.21)-- (-3,-1);
				\draw [line width=0.7pt] (-3,-2.21)-- (-1,-2.21);
				\draw [line width=0.7pt] (-4.48,-2.21)-- (-3,-2.21);
				\draw [line width=0.7pt] (-3,-1)-- (-4.48,-2.21);
				\draw [line width=0.7pt] (-4.48,-1)-- (-3,-2.21);
				\draw [line width=0.7pt] (-1,-2.21)-- (0.5,-2.21);
				\begin{scriptsize}
				\draw [fill=black] (-2,0) circle (3pt);
				\draw[color=black] (-2,0.26) node {0};
				\draw [fill=black] (-3,-1) circle (3pt);
				\draw[color=black] (-3,-0.74) node {2};
				\draw [fill=black] (-1,-1) circle (3pt);
				\draw[color=black] (-1,-0.74) node {5};
				\draw [fill=black] (0.5,-1) circle (3pt);
				\draw[color=black] (0.5,-0.74) node {4};
				\draw [fill=black] (-4.48,-1) circle (3pt);
				\draw[color=black] (-4.48,-0.74) node {3};
				\draw [fill=black] (-4.48,-2.21) circle (3pt);
				\draw[color=black] (-4.48,-1.89) node {4};
				\draw [fill=black] (-3,-2.21) circle (3pt);
				\draw[color=black] (-3,-1.89) node {1};
				\draw [fill=black] (-1,-2.21) circle (3pt);
				\draw[color=black] (-1,-1.89) node {6};
				\draw [fill=black] (0.5,-2.21) circle (3pt);
				\draw[color=black] (0.5,-1.89) node {3};
				\end{scriptsize}
				\end{tikzpicture}
				\caption{\label{fig1}A $6$-$L(2,1)$-labeling of $M(P_4)$  }
			\end{minipage}%
			\begin{minipage} {.5\textwidth} 
				\centering
				\begin{tikzpicture}[scale=0.7] 
				\clip(-5.86,-2.32) rectangle (2,0.39);
				\draw [line width=0.7pt] (-2,0)-- (-3.4,-1);
				\draw [line width=0.7pt] (-2,0)-- (-5,-1);
				\draw [line width=0.7pt] (-2,0)-- (-0.6,-1);
				\draw [line width=0.7pt] (-2,0)-- (1,-1);
				\draw [line width=0.7pt] (1,-1)-- (-0.6,-2.2);
				\draw [line width=0.7pt] (-0.6,-1)-- (1,-2.2);
				\draw [line width=0.7pt] (-3.4,-2.2)-- (-0.6,-2.2);
				\draw [line width=0.7pt] (-5,-2.2)-- (-3.4,-2.2);
				\draw [line width=0.7pt] (-3.4,-1)-- (-5,-2.2);
				\draw [line width=0.7pt] (-5,-1)-- (-3.4,-2.2);
				\draw [line width=0.7pt] (-0.6,-2.2)-- (1,-2.2);
				\draw [line width=0.7pt] (-2,0)-- (-2,-1);
				\draw [line width=0.7pt] (-2,-1)-- (-0.6,-2.2);
				\draw [line width=0.7pt] (-2,-1)-- (-3.4,-2.2);
				\draw [line width=0.7pt] (-3.4,-1)-- (-2.0294336585856554,-2.2);
				\draw [line width=1.1pt] (-0.6,-1)-- (-2.0294336585856554,-2.2);
				\begin{scriptsize}
				\draw [fill=black] (-2,0) circle (3pt);
				\draw[color=black] (-2,0.26) node {0};
				\draw [fill=black] (-3.4,-1) circle (3pt);
				\draw[color=black] (-3.4,-0.74) node {2};
				\draw [fill=black] (-0.6,-1) circle (3pt);
				\draw[color=black] (-0.6,-0.74) node {4};
				\draw [fill=black] (1,-1) circle (3pt);
				\draw[color=black] (1,-0.74) node {6};
				\draw [fill=black] (-5,-1) circle (3pt);
				\draw[color=black] (-5,-0.74) node {3};
				\draw [fill=black] (-5,-2.2) circle (3pt);
				\draw[color=black] (-5,-1.94) node {6};
				\draw [fill=black] (-3.4,-2.2) circle (3pt);
				\draw[color=black] (-3.4,-1.94) node {1};
				\draw [fill=black] (-0.6,-2.2) circle (3pt);
				\draw[color=black] (-0.6,-1.94) node {3};
				\draw [fill=black] (1,-2.2) circle (3pt);
				\draw[color=black] (1,-1.94) node {1};
				\draw [fill=black] (-2.0294336585856554,-2.2) circle (3pt);
				\draw[color=black] (-2.0294336585856554,-1.94) node {7};
				\draw [fill=black] (-2,-1) circle (3pt);
				\draw[color=black] (-1.8,-0.8) node {5};
				\end{scriptsize}
				\end{tikzpicture}
				\caption{\label{fig2}A $7$-$L(2,1)$-labeling of $M(P_5)$  }
			\end{minipage}
		\end{figure}
		\par Suppose that $\lambda(M(P_5))=6$. Then there is   an $L(2,1)$-labeling $f$ of $M(P_5)$ using labels in the set $L=\{0,1,2,3,4,5,6\}$. Since $deg_{M}(u)=5$, by Lemma~\ref{lem2.6}, $f(u)=0$ or $6$, without loss of generality, we suppose that $f(u)=0$.  We denote $N(v)$ the open neighborhood of a vertex $v$, and $N^2(v)$ the set of all vertices at distance at most $2$ from a vertex $v$ in $M(P_5)$. We have $N(u)=\{v'_1,v'_2,v'_3,v'_4,v'_5\}$, and $d_{M}(v'_i,v'_j)=2$, for $ 1\leq i,j\leq 5$. So each vertex $v'_i$ receives a distinct label from the set $\{2,3,4,5,6\}$.	We have $N^2(v_3)=\{u,v'_1,v'_2,v'_3,v'_4,v'_5\}$, and each vertex in $N^2(v_3)$ having a distinct label in $\{0,2,3,4,5,6\}$, which leaves only the label 1 from $L$ available for $v_3$. Then $f(v_3)=1$. We have
		$N^2(v_2)=\{u,v_3,v'_1,v'_2,v'_3,v'_4\}$, each vertex in $N^2(v_2)$ having a distinct label from $L$. So $f(v_2)$=$f(v'_5)$.
		Also $N^2(v_4)=\{u,v_3,v'_2,v'_3,v'_4,v'_5\}$, so $f(v_4)=f(v'_1)$. 
		$N^2(v_1)$=$\{u,v_2,v_3,v'_1,v'_2,v'_3\}$ and each vertex in $N^2(v_1)$ having a distinct label from $L$ ($f(v_2)$=$f(v'_5)$). Then $f(v_1)=f(v'_4)$. Also  $N^2(v_5)=\{u,v_3,v_4,v'_3,v'_4,v'_5\}$, with $f(v_4)=f(v'_1)$, so $f(v_5)=f(v'_2)$.
		We have $N(v_3)=\{v_2,v_4,v'_2,v'_4\}$, with  $f(v_2)=f(v'_5)$,  $f(v_4)=f(v'_1)$, and $f(v_3)=1$. It follows that the labels assigned to   $v'_1,v'_2,v'_4$ and $v'_5$ must be greater or equal to 3. Hence $f(v'_3)=2$.
		We have $v_2$ and $v_4$ are adjacent to $v'_3$, $f(v_2)=f(v'_5)$,  $f(v_4)=f(v'_1)$, and $f(v'_3)=2$, so $f(v'_5),f(v'_1)\in \{4,5,6\}$. Since $v'_1$ is adjacent to $v_2$ and $f(v_2)=f(v'_5)$, we have $|f(v'_5)-f(v'_1)|\geq 2$. Therefore $f(v'_1),f(v'_5)\in \{4,6\}$, and $f(v'_2),f(v'_4)\in \{3,5\}$. Since $f(v_2)=f(v'_5)$, $f(v_1)=f(v'_4)$,  $f(v_4)=f(v'_1)$ and $f(v_5)=f(v'_2)$. It follows that $|f(v'_5)-f(v'_4)|\geq 2$, and $|f(v'_1)-f(v'_2)|\geq 2$, impossible.	Therefore $\lambda(M(P_5)) \geq 7$. Hence $\lambda(M(P_5))= 7$.
		\item For $n\geq 6$, we define a labeling $f$ on $V(M(P_n))$ as following.
		\par $f(u)=0$, $f(v'_1)=6$, $f(v'_2)=5$, $f(v'_3)=4$, $f(v'_4)=7$, $f(v'_5)=2$, $f(v'_6)=3$, and  $f(v'_i)=i+1$ if $i\geq 7$. 
		\par $f(v_1)=7$, $f(v_2)=1$, $f(v_3)=3$, $f(v_4)=6$, $f(v_5)=1$, $f(v_6)=4$, and for $i\geq 7$: 
		$f(v_i)=6$ if $i\equiv 1\,(\text{mod } 3)$, $f(v_i)=2$ if $i\equiv 2\,(\text{mod } 3)$, $f(v_i)=4$ if $i\equiv 0\,(\text{mod } 3)$.\par	
		The idea is to come up with a 7-$L(2,1)$-labeling of the subgraph induced by $H=\{u,v_i,v'_i : 1\leq i\leq 6 \}$  isomorphic to $M(P_6)$. Then if $i\geq 7$, assign each vertex copy $v'_i$  consecutive labels beginning with $8$, and label the vertices $v_i$ with labels $6,2,4$  for $i\equiv 1\,(\text{mod } 3)$, $i\equiv 2\, (\text{mod } 3)$, and $i\equiv 0\, (\text{mod } 3)$, respectively. This is an $L(2,1)$-labeling of $M(P_n)$ with span $n+1$. 
		Hence $\lambda(M(P_n))\leq n+1$, for $n\geq 6$. It follows from Theorem~\ref{th3.1}, that $\lambda(M(P_n))= n+1$ , for $n\geq 6$.
	\end{itemize}
	\vspace{-4\topsep}
	\end{proof}
Let $C_n$ be the graph cycle, with vertex set $V(C_n)=\{v_0,v_1,\ldots,v_{n-1}\}$ and edge set $E(C_n)=\{v_iv_{i+1(\text{mod }n)} : 0\leq i \leq n-1\}$, where the indices are taken modulo $n$. We denote $V(M( C_n))= V(C_n)\cup  \{v'_i : 1\leq i \leq n\} \cup \{u\}$, we have $E(M(C_n))=\{v_iv_{i+1(\text{mod }n)}, v_iv'_{i+1(\text{mod }n)},v'_iv_{i+1(\text{mod }n)} : 0\leq i \leq n-1 \} \cup \{v'_iu : 0\leq i \leq n-1\}$.
\begin{prop}\label{thm3.7}
	$$ \lambda(M(C_n))=\begin{cases}
	6 \hspace{35pt}if\,\, n=3, \\
	8 \hspace{35pt}if\,\, n=4, \\
	10 \hspace{30pt}if\,\, n=5, \\
	n+1 \hspace{15pt}if\,\,n \geq 6.	
	\end{cases}  $$		
\end{prop}
\begin{proof}
	\leavevmode
	\begin{itemize}
		\item For $3 \leq n \leq 5 $, since $diam(C_3)=1$,  $diam(C_4)=diam(C_5)=2$,  from Lemma~\ref{lem2.4}, $diam(M(C_3))=diam(M(C_4))=diam(M(C_5))=2$. By applying Theorem~\ref{th3.4}, we get that  $\lambda(M(C_3))=6$, $\lambda(M(C_4))=8$, and $ \lambda(M(C_5))=10$.		
		\item For $n\geq 6$, we have Figure~\ref{fig3}, Figure~\ref{fig4}, and Figure~\ref{fig5}, respectively present an $L(2,1)$-Labeling for $M(C_6)$, $M(C_7)$, and $M(C_8)$, respectively with span $7$, $8$, and $9$. It follows from the lower bound in Theorem~\ref{th3.1} that $ \lambda(M(C_6))=7$, $ \lambda(M(C_7))=8$, and $ \lambda(M(C_8))=9$.
		\begin{figure}[h]
			\centering
			\begin{tikzpicture}[scale=0.7]
			\clip(-9.46,-2.42) rectangle (9.9,2.37);
			\draw [line width=0.7pt] (0,0)-- (-0.8660254037844385,0.5);
			\draw [line width=0.6pt] (0,0)-- (0.8660254037844387,0.5);
			\draw [line width=0.6pt] (0,0)-- (0,1);
			\draw [line width=0.6pt] (0,0)-- (-0.8660254037844385,-0.5);
			\draw [line width=0.6pt] (0,0)-- (0.8660254037844387,-0.5);
			\draw [line width=0.6pt] (0,0)-- (0,-1);
			\draw [line width=0.6pt] (0.8660254037844387,-0.5)-- (0,-2);
			\draw [line width=0.6pt] (0.8660254037844387,-0.5)-- (1.7320508075688774,1);
			\draw [line width=0.6pt] (0.8660254037844387,0.5)-- (1.7320508075688774,-1);
			\draw [line width=0.6pt] (0.8660254037844387,0.5)-- (0,2);
			\draw [line width=0.6pt] (0,1)-- (1.7320508075688774,1);
			\draw [line width=0.6pt] (0,1)-- (-1.732050807568877,1);
			\draw [line width=0.6pt] (-0.8660254037844385,0.5)-- (0,2);
			\draw [line width=0.6pt] (-0.8660254037844385,0.5)-- (-1.732050807568877,-1);
			\draw [line width=0.6pt] (-0.8660254037844385,-0.5)-- (0,-2);
			\draw [line width=0.6pt] (-0.8660254037844385,-0.5)-- (-1.732050807568877,1);
			\draw [line width=0.6pt] (0,-1)-- (1.7320508075688774,-1);
			\draw [line width=0.6pt] (0,-1)-- (-1.732050807568877,-1);
			\draw [line width=0.6pt] (-1.732050807568877,-1)-- (0,-2);
			\draw [line width=0.6pt] (0,-2)-- (1.7320508075688774,-1);
			\draw [line width=0.6pt] (1.7320508075688774,-1)-- (1.7320508075688774,1);
			\draw [line width=0.6pt] (1.7320508075688774,1)-- (0,2);
			\draw [line width=0.6pt] (0,2)-- (-1.732050807568877,1);
			\draw [line width=0.6pt] (-1.732050807568877,1)-- (-1.732050807568877,-1);
			\begin{scriptsize}
			\draw [fill=black] (0,0) circle (3pt);
			\draw[color=black] (0.16,0.3) node {0};
			\draw [fill=black] (0,1) circle (3pt);
			\draw[color=black] (0,1.23) node {2};
			\draw [fill=black] (0.8660254037844387,0.5) circle (3pt);
			\draw[color=black] (1.02,0.65) node {7};
			\draw [fill=black] (-0.8660254037844385,0.5) circle (3pt);
			\draw[color=black] (-1.05,0.55) node {3};
			\draw [fill=black] (-0.8660254037844385,-0.5) circle (3pt);
			\draw[color=black] (-1.05,-0.6) node {6};
			\draw [fill=black] (0,-1) circle (2.5pt);
			\draw[color=black] (0,-1.25) node {5};
			\draw [fill=black] (0.8660254037844387,-0.5) circle (3pt);
			\draw[color=black] (1.05,-0.6) node {4};
			\draw [fill=black] (1.7320508075688774,1) circle (3pt);
			\draw[color=black] (1.85,1.2) node {6};
			\draw [fill=black] (0,2) circle (2.5pt);
			\draw[color=black] (0.05,2.23) node {$1$};
			\draw [fill=black] (-1.732050807568877,1) circle (3pt);
			\draw[color=black] (-1.9,1.2) node {$4$};
			\draw [fill=black] (-1.732050807568877,-1) circle (3pt);
			\draw[color=black] (-2.02,-1.02) node {$7$};
			\draw [fill=black] (0,-2) circle (2.5pt);
			\draw[color=black] (-0.02,-2.3) node {$1$};
			\draw [fill=black] (1.7320508075688774,-1) circle (3pt);
			\draw[color=black] (2,-1) node {$3$};
			\end{scriptsize}
			\end{tikzpicture}
			\caption{\label{fig3}A $7$-$L(2,1)$-labeling of $M(C_6)$  }
		\end{figure}
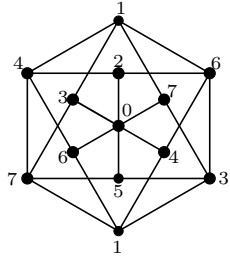
		\begin{figure}[h]
			\centering
			\begin{minipage}{.5\textwidth}
				\centering
				\begin{tikzpicture}[scale=0.55,rotate=-12]
				\clip(-3.2,-3.3) rectangle (3.5,3.4);
				
				\draw [line width=0.6pt] (1.8704694055762008,2.3454944474040893)-- (-0.667562801868943,2.924783736545471);
				\draw [line width=0.6pt] (-0.667562801868943,2.924783736545471)-- (-2.702906603707257,1.3016512173526746);
				\draw [line width=0.6pt] (-2.702906603707257,1.3016512173526746)-- (-2.7029066037072575,-1.301651217352674);
				\draw [line width=0.6pt] (-2.7029066037072575,-1.301651217352674)-- (-0.6675628018689438,-2.924783736545471);
				\draw [line width=0.6pt] (-0.6675628018689438,-2.924783736545471)-- (1.8704694055762001,-2.3454944474040897);
				\draw [line width=0.6pt] (1.8704694055762001,-2.3454944474040897)-- (3,0);
				\draw [line width=0.6pt] (3,0)-- (1.8704694055762008,2.3454944474040893);
				\draw [line width=0.6pt] (0.6234898018587336,0.7818314824680298)-- (0.6234898018587336,0.7818314824680298);
				\draw [line width=0.6pt] (-0.22252093395631434,0.9749279121818236)-- (-0.22252093395631434,0.9749279121818236);
				\draw [line width=0.6pt] (-0.900968867902419,0.43388373911755823)-- (-0.900968867902419,0.43388373911755823);
				\draw [line width=0.6pt] (-0.9009688679024191,-0.433883739117558)-- (-0.9009688679024191,-0.433883739117558);
				\draw [line width=0.6pt] (-0.2225209339563146,-0.9749279121818236)-- (-0.2225209339563146,-0.9749279121818236);
				\draw [line width=0.6pt] (0.6234898018587334,-0.7818314824680299)-- (0.6234898018587334,-0.7818314824680299);
				\draw [line width=0.6pt] (1,0)-- (1,0);
				\draw [line width=0.6pt] (0,0)-- (-0.22252093395631434,0.9749279121818236);
				\draw [line width=0.6pt] (0,0)-- (-0.900968867902419,0.43388373911755823);
				\draw [line width=0.6pt] (0,0)-- (-0.9009688679024191,-0.433883739117558);
				\draw [line width=0.6pt] (0,0)-- (-0.2225209339563146,-0.9749279121818236);
				\draw [line width=0.6pt] (0,0)-- (0.6234898018587334,-0.7818314824680299);
				\draw [line width=0.6pt] (0,0)-- (0.6234898018587336,0.7818314824680298);
				\draw [line width=0.6pt] (0,0)-- (0.6234898018587336,0.7818314824680298);
				\draw [line width=0.6pt] (0.6234898018587336,0.7818314824680298)-- (3,0);
				\draw [line width=0.6pt] (-0.6726224289868756,2.920748818561252)-- (0.6234898018587336,0.7818314824680298);
				\draw [line width=0.6pt] (-2.702906603707257,1.29)-- (-0.22252093395631434,0.9749279121818236);
				\draw [line width=0.6pt] (-0.6726224289868756,2.920748818561252)-- (-0.900968867902419,0.43388373911755823);
				\draw [line width=0.6pt] (-2.7029066037072575,-1.301651217352674)-- (-0.900968867902419,0.43388373911755823);
				\draw [line width=0.6pt] (-2.7029066037072575,-1.301651217352674)-- (-0.2225209339563146,-0.9749279121818236);
				\draw [line width=0.6pt] (-0.2225209339563146,-0.9749279121818236)-- (1.8704694055762001,-2.3454944474040897);
				\draw [line width=0.6pt] (-0.6897994987619294,-2.9070505625166714)-- (0.6234898018587334,-0.7818314824680299);
				\draw [line width=0.6pt] (0.6234898018587334,-0.7818314824680299)-- (3,0);
				\draw [line width=0.6pt] (1.8745555099842233,2.337009564172202)-- (-0.22252093395631434,0.9749279121818236);
				\draw [line width=0.6pt] (0,0)-- (1,0);
				\draw [line width=0.6pt] (1,0)-- (1.8704694055762001,-2.3454944474040897);
				\draw [line width=0.6pt] (1,0)-- (1.8745555099842233,2.337009564172202);
				\draw [line width=0.6pt] (-0.9009688679024191,-0.433883739117558)-- (-0.6897994987619294,-2.9070505625166714);
				\draw [line width=0.6pt] (-0.9009688679024191,-0.433883739117558)-- (-2.702906603707257,1.29);
				\begin{scriptsize}
				
				\draw [fill=black] (0,0) circle (3.5pt);
				\draw[color=black] (0.06,0.38) node {0};
				\draw [fill=black] (-0.22252093395631434,0.9749279121818236) circle (3.5pt);
				\draw[color=black] (-0.22,1.25) node {2};
				\draw [fill=black] (-0.900968867902419,0.43388373911755823) circle (3.5pt);
				\draw[color=black] (-1.1,0.55) node {5};
				\draw [fill=black] (0.6234898018587336,0.7818314824680298) circle (3.5pt);
				\draw [fill=black] (-0.9009688679024191,-0.433883739117558) circle (3.5pt);
				\draw[color=black] (-1.15,-0.5) node {4};
				\draw [fill=black] (-0.2225209339563146,-0.9749279121818236) circle (3pt);
				\draw[color=black] (-0.26,-1.25) node {7};
				\draw [fill=black] (0.6234898018587334,-0.7818314824680299) circle (3.5pt);
				\draw[color=black] (0.8,-1) node {8};
				\draw [fill=black] (0.6234898018587336,0.7818314824680298) circle (3.5pt);
				\draw [fill=black] (3,0) circle (3.5pt);
				\draw[color=black] (3.25,0.1) node {4};
				\draw [fill=black] (-0.6726224289868756,2.920748818561252) circle (3.5pt);
				\draw[color=black] (-0.72,3.21) node {1};
				\draw [fill=black] (-2.702906603707257,1.29) circle (3.5pt);
				\draw[color=black] (-3,1.4) node {8};
				\draw [fill=black] (-2.7029066037072575,-1.301651217352674) circle (3.5pt);
				\draw[color=black] (-2.94,-1.5) node {3};
				\draw [fill=black] (1.8704694055762001,-2.3454944474040897) circle (3.5pt);
				\draw[color=black] (2.12,-2.4) node {1};
				\draw [fill=black] (-0.6897994987619294,-2.9070505625166714) circle (3.5pt);
				\draw[color=black] (-0.92,-3.1) node {6};
				\draw [fill=black] (1.8745555099842233,2.337009564172202) circle (3.5pt);
				\draw[color=black] (2.04,2.6) node {7};
				\draw [fill=black] (0.6234898018587336,0.7818314824680298) circle (3.5pt);
				\draw [fill=black] (0.6234898018587336,0.7818314824680298) circle (3.5pt);
				\draw [fill=black] (0.6234898018587336,0.7818314824680298) circle (3.5pt);
				\draw [fill=black] (0.6234898018587336,0.7818314824680298) circle (3.5pt);
				\draw [fill=black] (0.6234898018587336,0.7818314824680298) circle (3.5pt);
				\draw [fill=black] (0.6234898018587336,0.7818314824680298) circle (3.5pt);
				\draw[color=black] (0.76,1.02) node {6};
				\draw [fill=black] (1,0) circle (3.5pt);
				\draw[color=black] (1.25,0.1) node {3};
				\end{scriptsize}
				\end{tikzpicture}
				\caption{\label{fig4}A $8$-$L(2,1)$-labeling of $M(C_7)$  }
			\end{minipage}%
			\begin{minipage}{.5\textwidth}
				\centering
				\begin{tikzpicture}[scale=0.55]
				\clip(-3.5,-4) rectangle (3.5,4);

				\draw [line width=0.6pt] (2.121320343559643,2.1213203435596424)-- (0,3);
				\draw [line width=0.6pt] (0,3)-- (-2.1213203435596424,2.121320343559643);
				\draw [line width=0.6pt] (-2.1213203435596424,2.121320343559643)-- (-3,0);
				\draw [line width=0.6pt] (-3,0)-- (-2.121320343559643,-2.1213203435596424);
				\draw [line width=0.6pt] (-2.121320343559643,-2.1213203435596424)-- (0,-3);
				\draw [line width=0.6pt] (0,-3)-- (2.121320343559642,-2.121320343559643);
				\draw [line width=0.6pt] (2.121320343559642,-2.121320343559643)-- (3,0);
				\draw [line width=0.6pt] (3,0)-- (2.121320343559643,2.1213203435596424);
				\draw [line width=0.6pt] (0.7071067811865476,0.7071067811865475)-- (0.7071067811865476,0.7071067811865475);
				\draw [line width=0.6pt] (0,1)-- (0,1);
				\draw [line width=0.6pt] (-0.7071067811865475,0.7071067811865476)-- (-0.7071067811865475,0.7071067811865476);
				\draw [line width=0.6pt] (-1,0)-- (-1,0);
				\draw [line width=0.6pt] (-0.7071067811865477,-0.7071067811865475)-- (-0.7071067811865477,-0.7071067811865475);
				\draw [line width=0.6pt] (0,-1)-- (0,-1);
				\draw [line width=0.6pt] (0.7071067811865474,-0.7071067811865477)-- (0.7071067811865474,-0.7071067811865477);
				\draw [line width=0.6pt] (1,0)-- (1,0);
				\draw [line width=0.6pt] (0,0)-- (0,1);
				\draw [line width=0.6pt] (0,0)-- (-0.7071067811865475,0.7071067811865476);
				\draw [line width=0.6pt] (0,0)-- (-1,0);
				\draw [line width=0.6pt] (0,0)-- (-0.7071067811865477,-0.7071067811865475);
				\draw [line width=0.6pt] (0,0)-- (0,-1);
				\draw [line width=0.6pt] (0,0)-- (0.7071067811865476,0.7071067811865475);
				\draw [line width=0.6pt] (0,0)-- (0.7071067811865476,0.7071067811865475);
				\draw [line width=0.6pt] (-0.00527335476512715,2.9978157049370795)-- (0.7071067811865476,0.7071067811865475);
				\draw [line width=0.6pt] (-0.00527335476512715,2.9978157049370795)-- (-0.7071067811865475,0.7071067811865476);
				\draw [line width=0.6pt] (-3,0)-- (-0.7071067811865475,0.7071067811865476);
				\draw [line width=0.6pt] (-3,0)-- (-0.7071067811865477,-0.7071067811865475);
				\draw [line width=0.6pt] (-0.7071067811865477,-0.7071067811865475)-- (0,-3);
				\draw [line width=0.6pt] (-2.130920163093153,-2.0981443290455077)-- (0,-1);
				\draw [line width=0.6pt] (0,-1)-- (2.121320343559642,-2.121320343559643);
				\draw [line width=0.6pt] (0,0)-- (0.7071067811865474,-0.7071067811865477);
				\draw [line width=0.6pt] (0.7071067811865474,-0.7071067811865477)-- (0,-3);
				\draw [line width=0.6pt] (0.7071067811865474,-0.7071067811865477)-- (2.996821354963606,-0.007673927956833637);
				\draw [line width=0.6pt] (-1,0)-- (-2.130920163093153,-2.0981443290455077);
				\draw [line width=0.6pt] (-1,0)-- (-2.1252529207063,2.1118262624771043);
				\draw [line width=0.6pt] (-2.1252529207063,2.1118262624771043)-- (0,1);
				\draw [line width=0.6pt] (2.121320343559643,2.1213203435596424)-- (0,1);
				\draw [line width=0.6pt] (0.7071067811865476,0.7071067811865475)-- (2.996821354963606,-0.007673927956833637);
				\draw [line width=0.6pt] (0,0)-- (1,0);
				\draw [line width=0.6pt] (1,0)-- (2.121320343559643,2.1213203435596424);
				\draw [line width=0.6pt] (1,0)-- (2.121320343559642,-2.121320343559643);
				\begin{scriptsize}
				
				\draw [fill=black] (0,0) circle (3.5pt);
				\draw[color=black] (0.15,0.4) node {0};
				\draw [fill=black] (0,1) circle (3.5pt);
				\draw[color=black] (0,1.25) node {2};
				\draw [fill=black] (-0.7071067811865475,0.7071067811865476) circle (3.5pt);
				\draw[color=black] (-0.9,0.9) node {3};
				\draw [fill=black] (0.7071067811865476,0.7071067811865475) circle (3.5pt);
				\draw [fill=black] (-1,0) circle (3.5pt);
				\draw[color=black] (-1.25,0) node {4};
				\draw [fill=black] (-0.7071067811865477,-0.7071067811865475) circle (3.5pt);
				\draw[color=black] (-0.9,-0.9) node {5};
				\draw [fill=black] (0,-1) circle (3.5pt);
				\draw[color=black] (0,-1.3) node {9};
				\draw [fill=black] (0.7071067811865476,0.7071067811865475) circle (3.5pt);
				\draw [fill=black] (2.121320343559642,-2.121320343559643) circle (3.5pt);
				\draw[color=black] (2.34,-2.22) node {4};
				\draw [fill=black] (-0.00527335476512715,2.9978157049370795) circle (3.5pt);
				\draw[color=black] (-0.04,3.25) node {9};
				\draw [fill=black] (-2.1252529207063,2.1118262624771043) circle (3.5pt);
				\draw[color=black] (-2.34,2.3) node {7};
				\draw [fill=black] (-3,0) circle (3.5pt);
				\draw[color=black] (-3.26,0.14) node {1};
				\draw [fill=black] (0,-3) circle (3.5pt);
				\draw[color=black] (0,-3.3) node {2};
				\draw [fill=black] (-2.130920163093153,-2.0981443290455077) circle (3.5pt);
				\draw[color=black] (-2.38,-2.14) node {6};
				\draw [fill=black] (2.996821354963606,-0.007673927956833637) circle (3.5pt);
				\draw[color=black] (3.2,0) node {1};
				\draw [fill=black] (0.7071067811865476,0.7071067811865475) circle (3.5pt);
				\draw [fill=black] (0.7071067811865476,0.7071067811865475) circle (3.5pt);
				\draw [fill=black] (0.7071067811865476,0.7071067811865475) circle (3.5pt);
				\draw [fill=black] (0.7071067811865476,0.7071067811865475) circle (3.5pt);
				\draw [fill=black] (0.7071067811865476,0.7071067811865475) circle (3.5pt);
				\draw [fill=black] (0.7071067811865476,0.7071067811865475) circle (3.5pt);
				\draw[color=black] (0.9,0.9) node {6};
				\draw [fill=black] (0.7071067811865474,-0.7071067811865477) circle (3.5pt);
				\draw[color=black] (0.9,-0.9) node {8};
				\draw [fill=black] (2.121320343559643,2.1213203435596424) circle (3.5pt);
				\draw[color=black] (2.35,2.3) node {5};
				\draw [fill=black] (1,0) circle (3.5pt);
				\draw[color=black] (1.25,0) node {7};
				\end{scriptsize}
				\end{tikzpicture}
				\caption{\label{fig5}A $9$-$L(2,1)$-labeling of $M(C_8)$  }
			\end{minipage}
		\end{figure}
		\par For $n\geq 9$, we partition the vertex set $V(C_n)$ into cliques in  $\overline{C^2_n}$ as following. 	
		\par If $n\equiv 0\,(\text{mod }3)$, for $0\leq i\leq \frac{n}{3}-1$, the sets  $S_i=\{v_i,v_{i+\frac{n}{3}},v_{i+2\frac{n}{3}}  \}$ form disjoint cliques of order $3$ in $\overline{C^2_n}$. We have $V(C_n)=\cup_{i=0} S_i$.  
		\par If $n\equiv 1\,(\text{mod }3)$, for $0\leq i\leq \lfloor \frac{n}{3}\rfloor-1$, the sets  $S_i=\{v_i,v_{i+\lfloor \frac{n}{3}\rfloor},v_{i+2\lfloor \frac{n}{3}\rfloor}  \}$ form disjoint cliques of order $3$ in $\overline{C^2_n}$. We have  $V(C_n)=\cup_{i=0} S_i \cup \{v_{n-1}\}$. 
		\par If $n\equiv 2\,(\text{mod }3)$, for  $1\leq i\leq \lceil \frac{n}{3}\rceil-1$, the sets  $S_i=\{v_i,v_{i+\lceil \frac{n}{3}\rceil},v_{i+2\lceil \frac{n}{3}\rceil-1}  \}$ form disjoint cliques of order $3$ in $\overline{C^2_n}$, and $v_0v_{\lceil \frac{n}{3}\rceil}$ is an edge in $\overline{C^2_n}$. We have $V(C_n)=\cup_{i=1} S_i \cup\{v_0,v_{\lceil \frac{n}{3}\rceil} \}$.\par
		The cycle $C_n$ in the three cases verifies the condition in Lemma~\ref{lemm3.1}. Hence    $\lambda(M(C_n))=n+1$ for  $n\geq 6$.
	\end{itemize}
	\vspace{-4\topsep}
	\end{proof}
For a connected graph $G$ of order $n$, in Theorem~\ref{th3.1} we have $\lambda(M(G))\geq n+1$. It means that for any fixed positive integer $k$, there are finitely many connected graphs having $\lambda(M(G))=k$. In the following we characterize the connected graphs with $\lambda(M(G))$ equal to $4,6\text{ and }7$, these are the smallest possible values for the $\lambda$-number of the Mycielski of any non-trivial connected graph.
\begin{cor}\label{cor3.6}
	For a connected graph $G$, we have\\
	 $1)$	$\lambda(M(G))=4$   if and only if  $G$ is $K_2$,\\ 
	 $2)$	$\lambda(M(G))=6$   if and only if  $G \in \{ P_3, P_4, C_3  \}$,\\
	$3)$	$\lambda(M(G))=7$   if and only if  $G \in \{ P_5, P_6, C_6  \}$.	
\end{cor}
\begin{proof}
	From Theorem~\ref{th3.1}, for a connected graph $G$ of order $n$ and maximum degree $\bigtriangleup$, we have $\lambda(M(G)) \geq max(n+1,2(\bigtriangleup+1))$. This means if $\bigtriangleup\geq 3$, then $\lambda(M(G))\geq 8$. The only connected graph with $\bigtriangleup=1$ is $K_2$, and we have from Theorem~\ref{th3.4},  $\lambda(M(K_2))=4$. If $\bigtriangleup=2$, then $G$ is either a path or a cycle, from Theorem~\ref{th3.1}  we have $\lambda(M(G))\geq 6$. By using Proposition~\ref{th3.6} and Proposition~\ref{thm3.7}, we can conclude the results.
\end{proof}

\section{The iterated Mycielski of a graph $M^t(G)$}\label{sec4}
\subsection{Bounds for $\lambda(M^t(G))$}
\begin{theorem}\label{th4.1}
	If $G$ is a  graph of order $n\geq 2$ and maximum degree $\bigtriangleup\geq 0$. For $t\geq 2$, we have $$2^{t-1}max(n+2, 2(\bigtriangleup+2))-2 \leq \lambda(M^t (G))\leq  (2^t-1)(n+1)+\lambda(G).$$
\end{theorem}
\begin{proof}
	For a graph $G$ of order $n\geq 2$ from Definition~\ref{def1}, we have $K_{1,n}$ is a subgraph of $M(G)$. Then by Observation~\ref{lem2.2}, $M^{t-1}(K_{1,n})$ is a subgraph of $M^t(G)$. Since $diam(K_{1,n})=2$, it follows from Lemma~\ref{lem2.5} and Lemma~\ref{lem2.7} that $ \lambda(M^t (G)) \geq \lambda(M^{t-1}(K_{1,n}))\geq |M^{t-1}(K_{1,n})|-1$.  By Lemma~\ref{lem2.1} $|M^{t-1}(K_{1,n})|=2^{t-1}(n+2)-1$, hence $ \lambda(M^t (G)) \geq 2^{t-1}(n+2)-2$, for $t\geq 2$. If $\bigtriangleup \geq 1$,  we have  $K_{1,\bigtriangleup}$ is a subgraph of $G$. By using the same arguments as preceding, we get that  $ \lambda(M^t (G))\geq 2^t(\bigtriangleup+2)-2$.
	\par On the other hand, for $t\geq2$, we have $M^t(G)=M(M^{t-1}(G))$. So by the upper bound of Theorem~\ref{th3.1},  $\lambda(M^t(G) )\leq (|M^{t-1}(G)|+1)+\lambda(M^{t-1}(G))=2^{t-1}(n+1)+\lambda(M^{t-1}(G))$. Recursively we get that $\lambda(M^t(G) )\leq \sum_{i=0}^{t-1}2^i(n+1)+\lambda(G)=(2^t-1)(n+1)+\lambda(G). $
\end{proof}
Notice that the lower bound $2^{t-1}(n+2)-2$ and the upper bound of Theorem~\ref{th4.1}, are true even for the trivial graph $K_1$. The upper bound coincides with the upper bound in Theorem~\ref{th3.1} for $t=1$. As a consequence we make the following observation.
\begin{obs}\label{obs4.1}
	If a graph $G$ of order $n$  has $\lambda(G)\leq n-1$, then for any $t\geq 1$, $\lambda(M^t(G))\leq |M^t(G) |-1=2^t(n+1)-2$,  there is equality if $G$ is of diameter two.
\end{obs}
Further, we denote $V^t=\{v^k_i : 1\leq i\leq n \text{ and } 0\leq k\leq 2^t-1 \}$, the set composed of the vertices of $V$ and all their copies in $M^t(G)$, where $v^1_i$ is the copy of $v^0_i$ in $M(G)$. $v^2_i$ and $v^3_i$ are respectively the copies of $v^0_i$ and $v^1_i$ in $M^2(G)$. $v^4_i,v^5_i,v^6_i,v^7_i$ are respectively the copies of  $v^0_i,v^1_i,v^2_i,v^3_i$ in $M^3(G)$ and so forth. In $M^t(G)$ for $0 \leq k \leq 2^{t-1}-1 $, we have $v_i^{2^{t-1}+k}$ is the exact copy of the vertex $v_i^k$ from $M^{t-1}(G)$.  For $t\geq 2$, let $U_t$ be the set of all the roots (i.e. roots and their consecutive copies in all levels) in $M^t(G)$. Recursively $U_t=U_{t-1}\cup U'_{t-1}\cup \{u_{t,0}\}$ and $|U_t|=2^t-1$. We denote the set of roots $U_t=\{u_{i,j}: 1\leq i\leq t \text{ and } 0\leq j\leq 2^{t-i}-1 \}$,  such that for example in $M^3(G)$, $u_{1,0}$ is the root of $M(G)$, $u_{1,1}$ the copy of $u_{1,0}$, and $u_{2,0}$ the root of $M^2(G)$. $u_{1,2}$, $u_{1,3}$, $u_{2,1}$ are respectively the copies of $u_{1,0}$, $u_{1,1}$, $u_{2,0}$, and $u_{3,0}$ is the root in $M^3(G)$, and so forth. Figure~\ref{fig6} illustrate an adjacency of a vertex and its copies $v^k_i$ in $M^2(G)$, with respect to the above ordering.
\begin{figure}[h]
	\centering
	\begin{tikzpicture} [scale=0.6]
	\clip(-11.5,-0.16) rectangle (0.5,5.57);
	\draw [line width=0.6pt] (-7,3.31)-- (-3.98,1.45);
	\draw [line width=0.6pt] (-4,3.33)-- (-7,1.47);
	\draw [line width=0.6pt] (-7,3.31)-- (-4,0);
	\draw [line width=0.6pt] (-4,3.33)-- (-7,0);
	\draw [line width=0.6pt] (-7,4.45)-- (-4,0);
	\draw [line width=0.6pt] (-4.02,4.43)-- (-7,0);
	\draw [line width=0.6pt] (-7,0)-- (-4,0);
	\draw [line width=0.6pt] (-4.02,4.43)-- (-2.5,2.35);
	\draw [line width=0.6pt] (-2.5,2.35)-- (-3.98,1.45);
	\draw [line width=0.6pt] (-1.58,3.95)-- (-3.98,1.45);
	\draw [line width=0.6pt] (-1.56,3.95)-- (-1.58,5.35);
	\draw [line width=0.6pt] (-1.56,5.35)-- (-4.02,4.43);
	\draw [line width=0.6pt] (-1.56,5.35)-- (-4,3.33);
	\draw [line width=0.6pt] (-7,1.47)-- (-4,0);
	\draw [line width=0.6pt] (-3.98,1.45)-- (-7,0);
	\begin{scriptsize}
	\draw [fill=black] (-7,3.31) circle (4pt);
	\draw[color=black] (-7.3,3.45) node {$v^2_j$};
	\draw [fill=black] (-7,4.45) circle (4pt);
	\draw[color=black] (-7.3,4.66) node {$v^3_j$};
	\draw [fill=black] (-7,1.47) circle (4pt);
	\draw[color=black] (-7.3,1.69) node {$v^1_j$};
	\draw [fill=black] (-7,0) circle (4pt);
	\draw[color=black] (-7.3,0.26) node {$v^0_j$};
	\draw [fill=black] (-4.02,4.43) circle (4pt);
	\draw[color=black] (-4.3,4.64) node {$v^3_i$};
	\draw [fill=black] (-4,3.33) circle (4pt);
	\draw[color=black] (-3.6,3.22) node {$v^2_i$};
	\draw [fill=black] (-3.98,1.45) circle (4pt);
	\draw[color=black] (-3.58,1.4) node {$v^1_i$};
	\draw [fill=black] (-4,0) circle (4pt);
	\draw[color=black] (-3.62,0.18) node {$v^0_i$};
	\draw [fill=black] (-2.5,2.35) circle (4pt);
	\draw[color=black] (-2,2.44) node {$u_{10}$};
	\draw [fill=black] (-1.58,3.95) circle (4pt);
	\draw[color=black] (-1,4.02) node {$u_{11}$};
	\draw [fill=black] (-1.56,5.35) circle (4pt);
	\draw[color=black] (-1,5.42) node {$u_{20}$};
	\end{scriptsize}
	\end{tikzpicture}
	\caption{\label{fig6} An example of an adjacency of the vertices  $v^k_i$  in  $M^2(G)$}	
\end{figure}
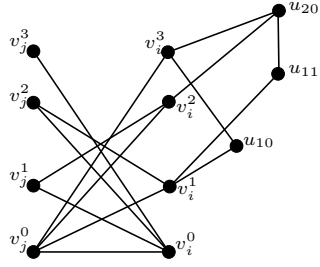
\begin{lemma}\label{lem4.1}
	If $d_G(v^0_i,v^0_j)\leq 2$,  then for any $t\geq 1$ and all $0\leq k,m\leq2^t-1$,  we have $d_{M^t}(v_i^k,v_j^m)\leq 2$, and if  $v^0_i$ is not an isolated vertex for $k\neq m$, we have $d_{M^t}(v_i^k,v_i^m)= 2$.  
\end{lemma}
\begin{proof}
	By using Lemma~\ref{lem2.4} inductively, we get the results. 
\end{proof}
The \textit{eccentricity} of a vertex $v$ in a graph $G$, being the greatest distance between $v$ and any other vertex in $G$. By Lemma~\ref{lem4.1}, if a vertex has eccentricity $1$ or $2$ in $G$, then the vertex and all its copies are of eccentricity $2$ in $M^t(G)$. In a graph $G$ without isolated vertices, we have from the definition of the Mycielski construction, the eccentricity of the root in $M(G)$ is $2$, so from above  the eccentricity of all the roots and their copies is $2$ in $M^t(G)$, for any $t\geq 1$.

\begin{prop}\label{prop4.1}
	If $G$ is a graph without isolated vertices of order $n$, with $k$ vertices of eccentricity $2$, for $t\geq 1$, we have $\lambda(M^t(G))\geq 2^{t-1}(n+k+2)-2.$
\end{prop}
\begin{proof}
	For $t\geq 1$, let $v^0_1,v^0_2,\ldots,v^0_k$ be the vertices of eccentricity $2$ in $G$. Let $V^{t-1}_i$ be the set composed of a vertex $v^0_i$ and all its copies in $M^{t-1}(G)$. In $M^t(G)$, by Lemma~\ref{lem4.1} and Definition~\ref{def1}, the vertices in $\cup_{i=1}^{k}V^{t-1}_i\cup V'_{t-1}\cup U_{t-1}\cup \{u_{t,0}\} $ are all within distance two, where $U_{t-1}$ is the set of roots and their copies in $M^{t-1}(G)$, $V'_{t-1}$ is the set of copies of the vertices of $M^{t-1}(G)$ in $M^t(G) $, and $u_{t,0}$ is the root of $M^t(G)$. Hence $\lambda(M^t(G))\geq\sum_{i=1}^{k}|V^{t-1}_i|+ |V'_{t-1}|+|U_{t-1}|=k2^{t-1}+2^{t-1}(n+1)-1+2^{t-1}-1= 2^{t-1}(n+k+2)-2.$ 
\end{proof}
For a graph $G$ of order $n$, by Proposition~\ref{prop4.1} if $\lambda(M(G))=n+1$, then $G$ has at most one vertex of eccentricity $2$. Also for $t\geq 2$, if $\lambda(M^t(G))=2^{t-1}(n+2)-2$, then no vertex in $G$ has eccentricity $2$. There exist graphs with one vertex of eccentricity $2$ and $\lambda(M(G))=n+1$,  Figure~\ref{fig7} illustrate a tree graph $T$ of order $9$ with one vertex of eccentricity $2$,  having $\lambda(M(T))=10$. Therefore from Proposition~\ref{prop4.1}, we have $\lambda(M(G))=n+1$ does not mean necessary that $\lambda(M^t(G))=2^{t-1}(n+2)-2$, for $t\geq 2$.
\begin{figure}[h]
	\centering
	\begin{tikzpicture}[scale=1]
	\clip(-8.14,-0.71) rectangle (4.06,3.1);
	\draw [line width=1.8pt] (-4,1)-- (-3,1);
	\draw [line width=1.8pt] (-3,1)-- (-2,1);
	\draw [line width=1.8pt] (-2,1)-- (-1,1);
	\draw [line width=1.8pt] (-1,1)-- (0,1);
	\draw [line width=1.8pt] (-2,1)-- (-1.58,0.21);
	\draw [line width=1.8pt] (-1.58,0.21)-- (-1.18,-0.59);
	\draw [line width=1.8pt] (-2.38,1.75)-- (-2,1);
	\draw [line width=1.8pt] (-2.78,2.61)-- (-2.38,1.75);
	\draw [line width=0.3pt] (-4.18,1.61)-- (-3,1);
	\draw [line width=0.3pt] (-3.22,1.57)-- (-4,1);
	\draw [line width=0.3pt] (-3.22,1.57)-- (-2,1);
	\draw [line width=0.3pt] (-2.2,2.99)-- (-2.38,1.75);
	\draw [line width=0.3pt] (-2.78,2.61)-- (-1.82,2.19);
	\draw [line width=0.3pt] (-2.38,1.75)-- (-1.48,1.45);
	\draw [line width=0.3pt] (-1.48,1.45)-- (-1.58,0.21);
	\draw [line width=0.3pt] (-1.48,1.45)-- (-3,1);
	\draw [line width=0.3pt] (-1.48,1.45)-- (-1,1);
	\draw [line width=0.3pt] (-0.42,1.49)-- (-2,1);
	\draw [line width=0.3pt] (-0.42,1.49)-- (0,1);
	\draw [line width=0.3pt] (0.5,1.49)-- (-1,1);
	\draw [line width=0.3pt] (-2,1)-- (-0.98,0.21);
	\draw [line width=0.3pt] (-0.98,0.21)-- (-1.18,-0.59);
	\draw [line width=0.3pt] (-1.82,2.19)-- (-2,1);
	\draw [line width=0.3pt] (-1.58,0.21)-- (-0.62,-0.57);
	\begin{scriptsize}
	\draw [fill=black] (-4,1) circle (3pt);
	\draw[color=black] (-3.94,0.76) node {$7$};
	\draw [fill=black] (-3,1) circle (3pt);
	\draw[color=black] (-2.9,0.76) node {$10$};
	\draw [fill=black] (-2,1) circle (3pt);
	\draw[color=black] (-2.05,0.75) node {$1$};
	\draw [fill=black] (0,1) circle (3pt);
	\draw[color=black] (0.15,0.8) node {$3$};
	\draw [fill=black] (-2.38,1.75) circle (3pt);
	\draw[color=black] (-2.58,1.77) node {$6$};
	\draw [fill=black] (-2.78,2.61) circle (3pt);
	\draw[color=black] (-2.98,2.6) node {$3$};
	\draw [fill=black] (-1.58,0.21) circle (3pt);
	\draw[color=black] (-1.8,0.23) node {$4$};
	\draw [fill=black] (-1.18,-0.59) circle (3pt);
	\draw[color=black] (-1.4,-0.55) node {$7$};
	\draw [fill=black] (-1,1) circle (3pt);
	\draw[color=black] (-0.84,0.84) node {$8$};
	\draw [fill=black] (-1.82,2.19) circle (3pt);
	\draw[color=black] (-1.65,2.33) node {$5$};
	\draw [fill=black] (-2.2,2.99) circle (3pt);
	\draw[color=black] (-1.99,3) node {$4$};
	\draw [fill=black] (-1.48,1.45) circle (3pt);
	\draw[color=black] (-1.32,1.6) node {$2$};
	\draw [fill=black] (-0.42,1.49) circle (3pt);
	\draw[color=black] (-0.26,1.6) node {$7$};
	\draw [fill=black] (0.5,1.49) circle (3pt);
	\draw[color=black] (0.76,1.6) node {$10$};
	\draw [fill=black] (-0.98,0.21) circle (3pt);
	\draw[color=black] (-0.82,0.37) node {$3$};
	\draw [fill=black] (-0.62,-0.57) circle (3pt);
	\draw[color=black] (-0.46,-0.4) node {$6$};
	\draw [fill=black] (-3.22,1.57) circle (3pt);
	\draw[color=black] (-3.2,1.8) node {$9$};
	\draw [fill=black] (-4.18,1.61) circle (3pt);
	\draw[color=black] (-4.18,1.83) node {$8$};
	\draw [fill=black] (-1,2.5) circle (3pt);
	\draw[color=black] (-0.8,2.6) node {0};
	\draw[color=black] (-1.21,2.6) node {$u$};
	\end{scriptsize}
	\end{tikzpicture}
	\caption{\label{fig7} A $10$-$L(2,1)$-Labeling of the Mycielski graph of a tree $T$ of order $9$.}
\end{figure}
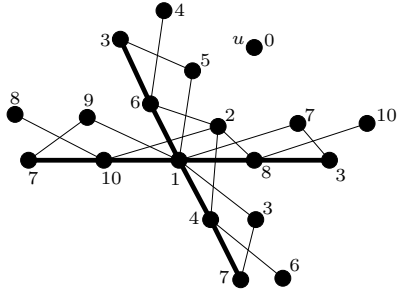

\subsection{Graphs with $\lambda(M^t(G))=2^t(n+1)-2$ }
Shao and Solis-Oba in \cite{shao}, gave bounds for the $\lambda$-number of some iterated Mycielski of complete graph $K_n$. In the following, we give the exact value of the $\lambda$-number of $M^t(K_n)$, for any $t\geq 2$.
\begin{theorem}\label{th4.2}
	For any $t\geq 2$ and $n\geq 2$, we have $\lambda(M^t(K_n))=2^t(n+1)-2$.	
\end{theorem}
\begin{proof}
	For $n\geq 2$, we have  $diam(K_n)=1$, so by Lemma~\ref{lem2.5} for any $t\geq 2$, we have $diam(M^t(K_n))=2$. Let $V^2=\{v^k_i :0\leq k \leq 3 \text{ and } 1\leq i\leq n \ \}$ be the set composed of the vertice of $V$ and all their consecutive copies in $M^2(K_n)$. Let $\chi_i$ with $1\leq i\leq n$, be a sequence of vertices in $M^2(K_n)$, where $\chi_i=v^2_iv^0_iv^1_i$ if $i$ is odd and $\chi_i=v^1_iv^0_iv^2_i$ if $i$ is even.
	We label the vertices of $M^2(K_n)$ using consecutive labels beginning with $0$, in the following order $\chi_1\chi_2\ldots \chi_nv^3_nv^3_{n-1}\ldots v^3_1u_{11}u_{10}u_{20}$.\par
	This does not violate the distance two conditions, since two consecutive vertices are either a vertex and its copy, or two vertices from the same level, which are successively at distance two. This leads to an $L(2,1)$-labeling of $M^2(K_n)$ with span $|M^2(K_n)|-1$. Since the diameter is $2$, then $\lambda(M^2(K_n))=|M^2(K_n)|-1$. From Observation~\ref{obs4.1} and Lemma~\ref{lem2.1}, we get $\lambda(M^t(K_n))=|M^t(K_n)|-1=2^t(n+1)-2$, for any $t\geq 2$.    	
\end{proof}
Since any graph $G$ of order $n\geq 2$ is a subgraph of the complete graph $K_n$, we can conclude that for $t\geq 2$, we have $\lambda(M^t(G))\leq |M^t(G)|-1=2^t(n+1)-2$. This could also be proven using Theorem~\ref{th3.4} by showing that for any graph $G$, the complement of the Mycielski $\overline{M}(G)$ has a perfect $4$-star matching, which means by Theorem~\ref{th3.4} $(a)$ that $\lambda(M^2(G))\leq |M^2(G)|-1$, then the result follows from Observation~\ref{obs4.1} for any $t\geq 2$.
\begin{cor}
	Let $G_1$ and $G_2$ be two graphs of the same order $|G_1|=|G_2|\geq 2$. For any $t\geq 2$, we have $\lambda(M^t(G_1))+2^t\leq \lambda(M^{t+1}(G_2))$. 	
\end{cor}
\begin{proof}
	For $t\geq 2$, let $G_1$ and $G_2$ be two graphs such that $|G_1|=|G_2|=n\geq 2$. By Theorem~\ref{th4.1} and Theorem~\ref{th4.2}, we have  $\lambda(M^t(G_1))\leq 2^t(n+1)-2$ and   $\lambda(M^{t+1}(G_2))\geq 2^t(n+2)-2$. Hence $\lambda(M^t(G_1))+2^t\leq \lambda(M^{t+1}(G_2)) $. 
\end{proof}
Let us denote $\overline{M^t}(G)$ the complement graph of $M^t(G)$, the close relation between Hamiltonicity and the $L(2,1)$-Labeling allow us to prove the following.
\begin{cor}
	For any graph $G$ and any $t\geq 2$, $\overline{M^t}(G)$ is a Hamiltonian graph.
\end{cor}
\begin{proof}
	Let $G$ be a graph of order $n$, first we show that $\overline{M^2}(G)$ is Hamiltonian.\par
	Let $\chi_i$ with $2\leq i\leq n$, be a sequence of vertices in $\overline{M^2}(G)$, where $\chi_i=v^2_iv^0_iv^1_i$ if $i$ is odd, and $\chi_i=v^1_iv^0_iv^2_i$ if $i$ is even. Take the vertices of $\overline{M^2}(G)$ in the following  order, $v^0_1v^1_1\chi_2\chi_3\ldots \chi_nv^3_nv^3_{n-1}\ldots v^3_1v^2_1u_{11}u_{10}u_{20}v^0_1$.
	\par Notice that this is similar to the  order proposed in Theorem~\ref{th4.2} for labeling $M^2(K_n)$. Since every two consecutive vertices are non-adjacent in $M^2(G)$, then the vertices of $\overline{M^2}(G)$ taken in the above order form a Hamiltonian cycle. Thus, for any graph $G$ we have $\overline{M^2}(G)$ is Hamiltonian. For $t\geq 2$, since $M^t(G)\cong M^2(M^{t-2}(G))$, then $\overline{M^t}(G)$ is a Hamiltonian graph for any $t\geq 2$.
\end{proof}
Next we characterize the graphs with $\lambda(M^t(G))=2^t(n+1)-2$, for $t\geq 2$.
\begin{theorem}\label{th4.4}
	Let $G$ be a graph of order $n\geq 2$. For $t\geq 2$, we have
	$\lambda(M^t(G))= 2^t(n+1)-2$ if and only if $G \cong K_n$ or $diam(G)=2$. 	
\end{theorem}
\begin{proof}
	For $t\geq2$, if $G\cong K_n$ by Theorem~\ref{th4.2} we have $\lambda(M^t(G))=2^t(n+1)-2$. If $diam(G)=2$, from Theorem~\ref{th4.2} we have $\lambda(M^t(G))\leq 2^t(n+1)-2$. By Lemma~\ref{lem2.5}, $diam(M^t(G))=2$, the vertices must be assigned distinct labels, hence  $\lambda(M^t(G))= 2^t(n+1)-2$.\par
	The converse, suppose that $G$ is a graph of order $n\geq 2$, with  $diam(G)\geq 3$. So there are at least two vertices at distance greater or equal to $3$, one from another. Without loss of generality, we suppose that $d_G(v^0_1,v^0_n)\geq 3$. For $t=2$, let $\chi_i$ with $2\leq i\leq n-1$, be a sequence of vertices in $M^2(G)$, where $\chi_i=v^2_iv^0_iv^1_i$ if $i$ is odd, and $\chi_i=v^1_iv^0_iv^2_i$ if $i$ is even. The labeling $f$  assigns consecutive labels to the vertices beginning with $0$ in the following order, $v^0_1v^1_1\chi_2\chi_3\ldots \chi_{n-1}v^3_{n-1}v^3_{n-2}\ldots v^3_1v^2_1$. \par This is similar to the order in Theorem~\ref{th4.2}. The maximum label assigned is $f(v^2_1)=4n-5$.   We have $d_G(v^0_1,v^0_n)\geq 3$, so by Lemma~\ref{lem2.4} we have $d_{M^2}(v^2_1,v^0_n)\geq 3$, and $d_{M^2}(v^2_1,v^1_n)=3$. We label  $f(v^0_n)=f(v^2_1)=4n-5$, $f(v^1_n)=4n-4$, $f(v^2_n)=4n-3$, $f(v^3_n)=4n-2$, $f(u_{11})=4n-1$, $f(u_{10})=4n$, $f(u_{20})=4n+1$. This is a valid $L(2,1)$-Labeling of $M^2(G)$ with span $4n+1$. Hence $\lambda(M^2(G))\leq 4n+1=4(n+1)-3$. From the upper bound of Theorem ~\ref{th3.1} and Theorem~\ref{th4.1}, for all $t\geq 3$, we have  $\lambda(M^t(G))\leq (2^{t-2}-1)(|M^2(G)|+1)+\lambda(M^2(G))$, since $|M^2(G)|=4(n+1)-1$, it follows that for all $t\geq 2$, $\lambda(M^t(G))\leq 2^t(n+1)-3$.~\end{proof}
\subsection{Graphs with $\lambda(M^t(G))=2^{t-1}(n+2)-2$ }
\begin{lemma}\label{lem4.2}
	Let $t\geq 2$ and $1\leq i,j\leq n$, for $1\leq k\leq 2^{t-1}-1 $, we have $d_{M^t}(v_i^{k},v_j^{2^{t-1}+k})=2$, and for $2^{t-1}+1\leq k \leq 2^t-1$, we have $d_{M^t}(v_i^{k},v_j^{2^{t-1}-1})=2$.
\end{lemma}
\begin{proof}
	For $1\leq k\leq 2^{t-1}-1 $, we have $v_j^{2^{t-1}+k}$ is the copy of $v_j^k$ in $M^t(G)$. Since $d_{M^{t-1}}(v_i^k,v_j^k)=2$, by Lemma~\ref{lem2.4} we have $d_{M^t}(v_i^{k},v_j^{2^{t-1}+k})=2$.\par
	For $t\geq 2$, $v^3_i$ is the copy of $v^1_i$. So by Lemma~\ref{lem2.4} $d_{M^2}(v_i^3,v_j^1)=2$. Since $d_{M^2}(v_i^3,v_j^2)=2$,  by using Lemma~\ref{lem2.4} inductively,  we can show that  for $2^{t-1}+1\leq k \leq 2^t-1$, we have $d_{M^t}(v_i^{k},v_j^{2^{t-1}-1})=2$.   
\end{proof} 
\begin{lemma}\label{lem4.3}
	If $v^0_i$ and $v^0_j$ are not isolated vertices, for $0\leq k\leq 2^{t-1}-1 $, we have $d_{M^t}(v_i^{k},v_j^{2^t-k-1})=min(3,d_G(v_i^0,v_j^0)) $.
\end{lemma}
\begin{proof}
	We have $v_i^{2^t-k-1}$ is the copy of $v_i^{2^{t-1}-k-1}$ in $M^t(G)$, by Lemma~\ref{lem2.4}, we have $d_{M^t}(v_i^k,v_j^{2^t-k-1})=min(3,d_{M^{t-1}}(v_i^k,v_j^{2^{t-1}-k-1})) $. If $0\leq k\leq 2^{t-2}-1$, we have $d_{M^{t-1}}(v_i^k,v_j^{2^{t-1}-k-1})=min(3,d_{M^{t-2}}(v_i^k,v_j^{2^{t-2}-k-1})) $. Otherwise, if $2^{t-2}\leq k\leq 2^{t-1}-1$, by symmetry $k=2^{t-1}-m-1$ where $0\leq m\leq 2^{t-2}-1$, so  $d_{M^{t-1}}(v_i^k,v_j^{2^{t-1}-k-1})=d_{M^{t-1}}(v_i^{2^{t-1}-m-1},v_j^m)=min(3,d_{M^{t-2}}(v_i^{2^{t-2}-m-1},v_j^m)) $. By recursively using Lemma~\ref{lem2.4}, we get $d_{M^t}(v_i^{k},v_j^{2^t-k-1})=min(3,d_G(v_i^0,v_j^0))$.~ \end{proof} 
In the case where $v^0_i$ or $v^0_j$ are isolated vertices,  for $1\leq k\leq 2^{t-1}-1$, we have $d_{M^t}(v_i^{k},v_j^{2^t-k-1})=3$.\par 
The direct product $G\times K_2$, called the \textit{canonical double cover} (or \textit{Kronecker double cover}) is a bipartite graph with two partition sets $X=V\times \{x\}$ and $Y=V\times \{y\}$, where $(v_i,x)(v_j,y) \in E(G\times K_2)$ if and only if $v_iv_j\in E(G)$.\par
From Lemma~\ref{lem4.3}, $v_i^{2^{t-1}-1}v_j^{2^{t-1}}\in E(M^t(G))$ if and only if $v_i^0v_j^0\in E(G)$. Since two copies of the same vertex or copies from the same level are non-adjacent, we have 
\begin{obs} \label{obs1}
	For $t\geq 2$, let $S=\{v_i^{2^{t-1}-1},v_i^{2^{t-1}} :1\leq i\leq n \}$. In $M^t(G)$, the subgraph induced by the vertices in $S$ is isomorphic to $G\times K_2$.
\end{obs}
A \textit{matching} in a graph $G$ is a collection of vertex-disjoint edges in $G$, a \textit{perfect matching} is a matching that covers all the vertices of $G$. The following theorem known as \textit{the Marriage Theorem}, gives a criterion for any bipartite graph $G=(X,Y)$ to have a perfect matching.
\begin{theorem}[(The Marriage Theorem)]\label{th4.5}
	Let $G=(X,Y)$ be a bipartite graph, then $G$ has a perfect matching  if and only if $|X|=|Y|$ and for any $S\subseteq X$, $|N_G(S)|\geq |S|$.
\end{theorem}
A \textit{$2$-matching}  of a graph $G$  is an assignment of weights $0$, $1$, or $2$ to the edges of $G$, such that the sum of weights of edges incident to any vertex in $G$ is less or equal to $2$ (see Chapter 6. in \cite{lovas}).  A $2$-matching of a graph $G$ can be seen as components with degree vertex at most $2$. The sum of weights in a $2$-matching is called the size. The maximum size of a $2$-matching is denoted by $\nu_2(G)$, which can be computed in polynomial time \cite{tutte}. A \textit{perfect $2$-matching} is a $2$-matching where the sum of weights incident to any vertex in $G$ is exactly $2$. W. Tutte in \cite{tutte}, provides a characterization for the existence of perfect $2$-matching of a graph.
\begin{theorem}\label{th4.6}\cite{tutte}
	A graph $G$ has a perfect $2$-matching  if and only if for any independent set $S\subseteq V$, $|N_G(S)|\geq |S|$. 
\end{theorem}
A perfect $2$-matching can be seen as a spanning subgraph in which each component is a single edge $K_2$ or a cycle, since every even cycle has a perfect matching, a graph with a perfect $2$-matching has a spanning subgraph in which each component is a single edge or an odd cycle. It is easy to see from the two preceding Theorem~\ref{th4.5} and Theorem~\ref{th4.6}, that the existence of perfect $2$-matching in a graph $G$ is equivalent to that $G\times K_2$ admits a perfect matching.
\begin{theorem}\label{th4.7}
	Let $G$ be a graph without isolated vertices of order $n\geq 2$. For $t\geq 2$, $\lambda(M^t(G))=2^{t-1}(n+2)-2$ if and only if for any $S\subseteq V$  $|D_2(S)|\geq |S|$, where $D_2(S)=\{x\in V :\exists v\in S, d_G(x,v)>2 \}$.	
\end{theorem}
\begin{proof}
	Let $G$ be a graph without isolated vertices of order $n\geq 2$, such that for $t \geq 2$, $\lambda(M^t(G))=2^{t-1}(n+2)-2$. Let $f$ be a $\lambda$-labeling of $M^t(G)$, using labels from the set  $L=\{0,\ldots,2^{t-1}(n+2)-2\}$. From Lemma~\ref{lem4.1}, we have $d_{M^t}(v^k_i,u) \leq 2$ and $d_{M^t}(u,u') \leq 2$,  for all $v^k_i\in V^t$ and all $u,u'\in U_t$. The roots are assigned distinct labels, different from the labels assigned to the vertices in $V^t$. So for $2^{t-1}\leq k \leq 2^t-1$, we have $f(v^k_i)\in L\setminus f(U_t)$ and $|L\setminus f(U_t)|=2^{t-1}n$. For $1\leq i,j \leq n$, we have  $d_{M^t}(v^k_i,v^m_j)=2$, where $2^{t-1}\leq k,m \leq 2^t-1$. It follows that the $2^{t-1}n$ vertices $v^k_i$ where $2^{t-1}\leq k \leq 2^t-1$, and $1\leq i \leq n$, have distinct labels and use all the labels in $L\setminus f(U_t)$. By Lemma~\ref{lem4.2}, we have $d_{M^t}(v^k_i,v^{2^{t-1}-1}_j)=2$, for $2^{t-1}+1\leq k \leq 2^t-1$. The only labels remaining in $L\setminus f(U_t)$, for the vertices $v^{2^{t-1}-1}_j$, are those assigned to the vertices $v_i^{2^{t-1}}$. Since $d_{M^t}(v^{2^{t-1}-1}_i,v^{2^{t-1}-1}_j)=2$ and $d_{M^t}(v^{2^{t-1}}_i,v^{2^{t-1}}_j)=2$, then $f(v^{2^{t-1}-1}_i)\neq f(v^{2^{t-1}-1}_j)$ and $f(v^{2^{t-1}}_i)\neq f(v^{2^{t-1}}_j)$. It follows that for any vertex $v^{2^{t-1}}_j$,   there is one and only one vertex $v^{2^{t-1}-1}_i$, such that $f(v^{2^{t-1}-1}_i)=f(v^{2^{t-1}}_j)$. Let $(v_i,x)$ and $(v_j,y)$, $1\leq i,j\leq n$ denote the vertices of $G\times K_2$, where $(v_i,x)(v_j,y)\in E(G\times K_2)$ if and only if  $v_i^0v_j^0 \in E(G)$. Let $M=\{(v_i,x)(v_j,y) : f(v^{2^{t-1}-1}_i)=f(v^{2^{t-1}}_j)\}$. Since $f(v^{2^{t-1}-1}_i)=f(v^{2^{t-1}}_j)$ means by Lemma~\ref{lem4.3}, that $d_G(v_i^0,v_j^0)\geq 3$. From Observation~\ref{obs1}, $M$ is a perfect matching of the graph $\overline{G^2}\times K_2$, then by Theorem~\ref{th4.5} we get the necessity. \par
	The converse, suppose that for any $S\subseteq V$, we have  $|D_2(S)|\geq |S|$. This means by Theorem~\ref{th4.6}, that the graph $\overline{G^2}$ has a perfect $2$-matching, which means that $\overline{G^2}$ has a spanning subgraph $H$, whose connected components are vertex-disjoint edges or odd cycles. Let  $E^1,E^2,\ldots,E^r$  be the $K_2$ components, and $C^1,C^2,\ldots,C^s$  the odd cycle components of $H$. Let us denote the vertices of $V$ as  $x^0_iy^0_i$ is the edge $E^i$ and $c^0_{1,i}c^0_{2,i}\ldots c^0_{n_i,i}$ is the odd cycle $C^i$, where $n_i=|C^i|$. We define an $L(2,1)$-Labeling $f$ to the vertices of $M^t(G)$ as follows.\par
	Suppose that $r\geq 2$, first we label the vertices $x^k_1,y^k_1$  with $0\leq k\leq 2^t-1$, where $x^k_1$ and $y^k_1$  are the vertices $x^0_1$ and $y^0_1$ and their consecutive copies. The labeling $f$ assigns in descending order the labels $2^{t-1}-1,2^{t-1}-2,\ldots,0$ respectively to $x^0_1,x^1_1,\ldots,x^{2^{t-1}-1}_1$ and the labels $2^t-1,2^t-2,\ldots,2^{t-1}$ respectively to  $x^{2^{t-1}}_1,x^{2^{t-1}+1}_1,\ldots,x^{2^t-1}_1$. Then assign the same list of consecutive labels, now in ascending order $0,1,\ldots,2^{t-1}-1$ respectively to the vertices $y^{2^{t-1}}_1,y^{2^{t-1}+1}_1,\ldots,y^{2^t-1}_1$ and the labels $2^{t-1},2^{t-1}+1,\ldots,2^t-1$ respectively to $y^0_1,y^1_1,\ldots,y^{2^{t-1}-1}_1$.
	\begin{itemize}
		\item For $0\leq k \leq 2^{t-1}-1$, $f(x^k_1)=2^{t-1}-k-1$, and for $2^{t-1}\leq k \leq 2^t-1$,  $f(x^k_1)=3\times 2^{t-1}-k-1$.
		\item For $0\leq k \leq 2^{t-1}-1$, $f(y^k_1)=k+2^{t-1}$, and for $2^{t-1}\leq k \leq 2^t-1$,  $f(y^k_1)=k-2^{t-1}$.
	\end{itemize}
	\par  
	We have $f(x^k_1)=f(y^m_1)$ if $m=2^t-k-1$. Since $x^0_1y^0_1\in E(\overline{G^2})$, then  $d_G(x^0_1,y^0_1)\geq 3$, so by Lemma~\ref{lem4.3} $d_{M^t}(x^k_1,y^{2^t-k-1}_1)=3$. Otherwise $f(x^k_1)\neq f(y^m_1)$, since $x^0_1$ and $y^0_1$ are not adjacent in $G$ we have  $d_{M^t}(x^k_1,y^m_1)\geq 2$, for all $0\leq k,m \leq 2^t-1$. Also $d_{M^t}(x^k_1,x^m_1)=d_{M^t}(y^k_1,y^m_1)=2$, $f(x^k_1)\neq f(x^m_1)$ and $f(y^k_1)\neq f(y^m_1)$. The smallest label is $f(x^{2^{t-1}-1}_1)=f(y^{2^{t-1}}_1)=0$, the maximum label is $f(x^{2^{t-1}}_1)=f(y^{2^{t-1}-1}_1)=2^t-1$.\par
	For $2\leq i\leq r$,  we have $d_G(x^0_i,y^0_i)\geq 3$, so a vertex in $E_{i-1}$ cannot be adjacent in $G$ to both $x^0_i$ and $y^0_i$. Since in every $E^i$ the vertices $x^0_i$ and $y^0_i$ are symmetric, we rearrange the vertices of each $E^i$ depending on the cases:
	\item  $i)$ If $x^0_{i-1}$ is adjacent in $G$ to a vertex in $E^i$, we consider without loss of generality that $x^0_{i-1}$ is adjacent to $y^0_i$.
	\item  $ii)$ If $x^0_{i-1}$ is not adjacent to $E^i$ and $y^0_{i-1}$ is adjacent, we let $d_G(y^0_{i-1},x^0_i)=1$. Otherwise the vertices in $E^{i-1}$ and $E^i$ are mutually non-adjacent. This means that $d_G(x^0_{i-1},x^0_i)\geq 2$, and $d_G(y^0_{i-1},y^0_i)\geq 2$, for all $2\leq i\leq r$.\par
	With respect to the above assumptions, we label the vertices $x^k_i$ and $y^k_i$ with $2\leq i\leq r$, as following.
	\begin{itemize}
		\item  For $2\leq i\leq r-1$, and $0\leq k\leq 2^t-1$, $f(x^k_i)=(i-1)2^t+f(x^k_1)$, and $f(y^k_i)=(i-1)2^t+f(y^k_1)$. 
		\item For $0\leq k\leq 2^{t-1}-1$, $f(x^k_r)=(r-1)2^t+f(x^k_1)$, and for $2^{t-1}\leq k\leq 2^t-1$, $f(x^k_r)=(r-1)2^t+k$. 
		\item For $0\leq k\leq 2^{t-1}-1$, $f(y^k_r)=r2^t-k-1$, and for  $2^{t-1}\leq k\leq 2^t-1$, $f(y^k_r)=(r-1)2^t+f(y^k_1)$.
	\end{itemize} 
	\par 
	The labeling $f$ uses distinct labels from $(i-1)2^t,\ldots, i2^t-1$, for every pair of $x^k_i$,$y^m_i$, where $m=2^t-k-1$,  by using the same pattern for $x^k_1$,$y^m_1$ (except for $x^k_r$,$y^k_r$). In the case where $r=1$, let for $0\leq k \leq 2^{t-1}-1$, $f(x^k_1)=2^{t-1}-k-1$, for $2^{t-1}\leq k\leq 2^t-1$, $f(x^k_1)=k$,  for $0\leq k\leq 2^{t-1}-1$, $f(y^k_1)=2^t-k-1$, and for $2^{t-1}\leq k \leq 2^t-1$,  $f(y^k_1)=k-2^{t-1}$.  The only vertices from two different components, with the difference between the labels equal to $1$, are for $x^{2^{t-1}}_{i-1}$ and $y^{2^{t-1}-1}_{i-1}$, with both $x^{2^{t-1}-1}_i$ and $y^{2^{t-1}}_i$. This does not violate the distance two conditions, since $d_G(x^0_{i-1},x^0_i)\geq 2$, and $d_G(y^0_{i-1},y^0_i)\geq 2$, for all $2\leq i\leq r$.
	The maximum label assigned is $f(x^{2^t-1}_r)=f(y^0_r)=r2^t-1$.\par
	If $s\geq 1$, next we label the vertices of the odd cycle components $C^i$.
	We make the following claim.
	\begin{claim}
		For a vertex $v$ in $G$ not in the odd cycle component $C^i=c^0_{1,i}c^0_{2,i}\ldots c^0_{n_i,i}$, there is at least one edge  $c^0_{p,i}c^0_{q,i}\in C^i$, such that $v$ is not adjacent in $G$ to both $c^0_{p,i}$ and $c^0_{q,i}$.
	\end{claim}
	\begin{proof}
		We prove this by using contradiction, we suppose that $v$ is adjacent to at least one endpoint of any $c^0_{p,i}c^0_{q,i}\in C^i$. We may assume that $v$ is adjacent to $c^0_{1,i}$. Since $d_G(c^0_{1,i},c^0_{2,i})\geq 3$, $v$ is not adjacent to $c^0_{2,i}$, so $v$ is adjacent to $c^0_{3,i}$, and so forth. Hence, if $j$ is odd $v$ is adjacent to  $c^0_{j,i}$, and if $j$ is even $v$ is not adjacent to $c^0_{j,i}$. Since $v$ is adjacent to $c^0_{1,i}$, then $v$ is not adjacent to $c^0_{n_i,i}$. It follows that $n_i$ is even, a contradiction.~\end{proof}
	Since the cycles $C^i$ are symmetric, we may consider that $d_G(y_r,c^0_{1,1})\geq 2$, and $d_G(y_r,c^0_{n_1,1})\geq 2$, and for $1\leq i \leq s-1$,  $d_G(c^0_{n_i,i},c^0_{1,i+1})\geq 2$, and $d_G(c^0_{n_i,i},c^0_{n_{i+1},i+1})\geq 2$. We label the vertices $c^k_{j,i}$ where  $1\leq j\leq n_i$, $1\leq i\leq s$ and $0\leq k\leq 2^t-1$, with respect to the above assumptions. 
	\begin{itemize}
		\item For $0\leq k \leq 2^{t-1}-1$, $f(c^k_{1,1})=r2^t+2^{t-1}-k-1$, and for $2^{t-1}\leq k \leq 2^t-1$,  $f(c^k_{1,1})=r2^t+k$. 
		\item For $2\leq j \leq n_1-1$ and all $0\leq k \leq 2^t-1$, $f(c^k_{j,1})=f(c^k_{1,1})+(j-1)2^{t-1}$.
		\item For $0 \leq k \leq 2^{t-1}-1$, $f(c^k_{n_1,1})=f(c^k_{1,1})+(n_1-1)2^{t-1}$, and for $2^{t-1}\leq k \leq 2^t-1$,   $f(c^k_{n_1,1})=f(c^{2^t-k-1}_{1,1})$.	
	\end{itemize}	
	\par
	The smallest label for the vertices $c^k_{i,1}$ is $f(c^{2^{t-1}-1}_{1,1})=f(c^{2^{t-1}}_{n_1,1})=r2^t$ and the maximum is $f(c^0_{n_1,1})=f(c^{2^t-1}_{n_1-1,1})=r2^t+n_12^{t-1}-1$.  Now let $\varphi_i=r2^t+\sum_{j=1}^{i-1}n_j2^{t-1}$. For $2\leq i\leq s$, we label $f(c^{2^{t-1}-1}_{1,i})=f(c^{2^{t-1}}_{n_i,i})=\varphi_i$, then
	\begin{itemize}
		\item For $0\leq k \leq 2^{t-1}-1$, $f(c^k_{1,i})=\varphi_i+2^{t-1}-k-1$, and for $2^{t-1}\leq k \leq 2^t-1$, $f(c^k_{1,i})=\varphi_i+k$.
		\item For $2\leq j \leq n_i-1$ and all $0\leq k \leq 2^t-1$, $f(c^k_{j,i})=f(c^k_{1,i})+(j-1)2^{t-1}$.\par     
		\item For  $0 \leq k \leq 2^{t-1}-1$, 	 $f(c^k_{n_i,i})=f(c^k_{1,i})+(n_i-1)2^{t-1}$, and for $2^{t-1}\leq k \leq 2^t-1$,  $f(c^k_{n_i,i})=f(c^{2^t-k-1}_{1,i})$.	
	\end{itemize}	
	\par
	The labeling $f$ uses $n_i2^{t-1}$ distinct labels for the $n_i2^t$ vertices of each component $C^i$ and their copies. For $0 \leq k \leq 2^{t-1}-1$, we have $f(c^k_{1,i})=f(c^{2^t-k-1}_{n_i,i})$, and for $2\leq j\leq n_i$ $f(c^k_{j,i})=f(c^{2^t-k-1}_{j-1,i})$. It is possible, since $d_G(c^0_{j,i},c^0_{j-1,i})\geq 3$, which means by Lemma~\ref{lem4.3} that $d_{M^t}(c^k_{j,i},c^{2^t-k-1}_{j-1,i})=3$. For two vertices  $c^k_{j,i}$, $c^m_{l,i}$ from the same component,  the difference between the labels is equal to $1$ in the following cases: $i)$ The vertices are copies of the same vertex, or if $ 2^{t-1}\leq k,m\leq 2^t-1$, in those two cases $d_{M^t}(c^k_{j,i},c^m_{l,i})=2$.  $ii)$  For $l=j+1$, we have $d_G(c^0_{j,i},c^0_{j+1,i})\geq 3$, then  $d_{M^t}(c^k_{j,i},c^m_{j+1,i})\geq 2$. 
	$iii)$ If $l=j+2$, $k=2^t-1$ and $m=2^{t-1}-1$, we have from Lemma~\ref{lem4.2} $d_{M^t}(c^{2^t-1}_{j,i},c^{2^{t-1}-1}_{l,i})=2$. The difference between the labels is equal to $1$ for vertices from two different odd cycle components, only occur for $c^0_{n_i,i}$ and $c^{2^t-1}_{n_{i-1},i}$ with $c^{2^{t-1}-1}_{1,i+1}$ and  $c^{2^{t-1}}_{n_{i+1},i+1}$, for $1\leq i \leq s-1$, we have  $d_G(c^0_{n_i,i},c^0_{n_{i+1},i+1})\geq 2$ and $d_G(c^0_{n_i,i},c^0_{1,i+1})\geq 2$,  from Lemma~\ref{lem4.2} the vertices are at distance greater or equal $2$ in $M^t(G)$.
	The maximum label assigned is $f(c^0_{n_s,s})=f(c^{2^t-1}_{n_s-1,s})=r2^t+\sum_{j=1}^{s}n_j2^{t-1}-1=n2^{t-1}-1$.

	\par
	We finally label the remaining $2^t-1$ roots with consecutive labels beginning with the label $n2^{t-1}$ in the following order $$u_{1,2^{t-1}-1}u_{1,2^{t-1}-2}\ldots u_{1,0}u_{2,2^{t-2}-1}u_{2,2^{t-2}-2}\ldots u_{2,0}u_{3,2^{t-3}-1}\ldots u_{t,0} $$		
	\begin{figure}[h]
		
		\centering
		
		\begin{tikzpicture} [scale=0.7, rotate=0.5]
		\clip(-2.3,-1.52) rectangle (13,5.85);
		\draw [line width=0.6pt] (-2,2)-- (-1,0);
		\draw [line width=0.6pt] (1,2)-- (2,0);
		\draw [line width=0.6pt] (4,2)-- (5,0);
		\draw [line width=0.6pt] (5,2)-- (6,0);
		\draw [line width=0.6pt] (6,2)-- (4,0);
		\draw [line width=0.6pt] (8,2)-- (9,0);
		\draw [line width=0.6pt] (9,2)-- (10,0);
		\draw [line width=0.6pt] (10,2)-- (11,0);
		\draw [line width=0.6pt] (11,2)-- (12,0);
		\draw [line width=0.6pt] (8,0)-- (12,2);
		\draw [line width=0.6pt] (-1,2)-- (-2,0);
		\draw [line width=0.6pt] (2,2)-- (1,0);
		\begin{scriptsize}
		\draw [fill=black] (-2,3) circle (2.5pt);
		\draw[color=black] (-1.9,3.25) node {2};
		\draw [fill=black] (-2,2) circle (2.5pt);
		\draw[color=black] (-1.86,2.25) node {3};
		\draw [fill=black] (-2,0) circle (2.5pt);
		\draw[color=black] (-2.1,0.25) node {0};
		\draw [fill=black] (-2,-1) circle (2.5pt);
		\draw[color=black] (-1.88,-0.75) node {1};
		\draw [fill=black] (-2,-1) circle (2.5pt);
		\draw[color=black] (-2,-1.3) node {$x^0_1$};
		\draw [fill=black] (-1,-1) circle (2.5pt);
		\draw[color=black] (-0.86,-0.7) node {2};
		\draw [fill=black] (-1,-1) circle (2.5pt);
		\draw[color=black] (-0.86,-1.3) node {$y^0_1$};
		\draw [fill=black] (-1,0) circle (2.5pt);
		\draw[color=black] (-0.82,0.25) node {3};
		\draw [fill=black] (-1,2) circle (2.5pt);
		\draw[color=black] (-0.76,2.25) node {0};
		\draw [fill=black] (-1,3) circle (2.5pt);
		\draw[color=black] (-0.74,3.29) node {1};
		\draw [fill=black] (1,3) circle (2.5pt);
		\draw[color=black] (1.16,3.25) node {7};
		\draw [fill=black] (1,2) circle (2.5pt);
		\draw[color=black] (1.16,2.25) node {6};
		\draw [fill=black] (1,0) circle (2.5pt);
		\draw[color=black] (0.9,0.25) node {4};
		\draw [fill=black] (1,-1) circle (2.5pt);
		\draw[color=black] (1.1,-0.7) node {5};
		\draw [fill=black] (1,-1) circle (2.5pt);
		\draw[color=black] (1.1,-1.3) node {$x^0_2$};
		\draw [fill=black] (2,-1) circle (2.5pt);
		\draw[color=black] (2.1,-0.7) node {7};
		\draw [fill=black] (2,-1) circle (2.5pt);
		\draw[color=black] (2.1,-1.3) node {$y^0_2$};
		\draw [fill=black] (2,0) circle (2.5pt);
		\draw[color=black] (2.16,0.25) node {6};
		\draw [fill=black] (2,2) circle (2.5pt);
		\draw[color=black] (2.16,2.25) node {4};
		\draw [fill=black] (2,3) circle (2.5pt);
		\draw[color=black] (2.16,3.25) node {5};
		\draw [fill=black] (4,3) circle (2.5pt);
		\draw[color=black] (4.24,3.25) node {11};
		\draw [fill=black] (5,3) circle (2.5pt);
		\draw[color=black] (5.24,3.25) node {13};
		\draw [fill=black] (6,3) circle (2.5pt);
		\draw[color=black] (6.16,3.25) node {9};
		\draw [fill=black] (4,2) circle (2.5pt);
		\draw[color=black] (4.24,2.25) node {10};
		\draw [fill=black] (5,2) circle (2.5pt);
		\draw[color=black] (5.24,2.25) node {12};
		\draw [fill=black] (6,2) circle (2.5pt);
		\draw[color=black] (6.16,2.25) node {8};
		\draw [fill=black] (4,0) circle (2.5pt);
		\draw[color=black] (3.95,0.25) node {8};
		\draw [fill=black] (5,0) circle (2.5pt);
		\draw[color=black] (5.24,0.25) node {10};
		\draw [fill=black] (6,0) circle (2.5pt);
		\draw[color=black] (6.24,0.25) node {12};
		\draw [fill=black] (4,-1) circle (2.5pt);
		\draw[color=black] (4.1,-0.7) node {9};
		\draw [fill=black] (4,-1) circle (2.5pt);
		\draw[color=black] (4.16,-1.3) node {$c^0_{1,1}$};
		\draw [fill=black] (5,-1) circle (2.5pt);
		\draw[color=black] (5.1,-0.7) node {11};
		\draw [fill=black] (5,-1) circle (2.5pt);
		\draw[color=black] (5.2,-1.3) node {$c^0_{2,1}$};
		\draw [fill=black] (6,-1) circle (2.5pt);
		\draw[color=black] (6.1,-0.7) node {13};
		\draw [fill=black] (6,-1) circle (2.5pt);
		\draw[color=black] (6.2,-1.3) node {$c^0_{3,1}$};
		\draw [fill=black] (8,0) circle (2.5pt);
		\draw[color=black] (8,0.25) node {14};
		\draw [fill=black] (8,-1) circle (2.5pt);
		\draw[color=black] (8.1,-0.7) node {15};
		\draw [fill=black] (8,-1) circle (2.5pt);
		\draw[color=black] (8.2,-1.3) node {$c^0_{1,2}$};
		\draw [fill=black] (9,-1) circle (2.5pt);
		\draw[color=black] (9.1,-0.7) node {17};
		\draw [fill=black] (9,-1) circle (2.5pt);
		\draw[color=black] (9.2,-1.3) node {$c^0_{2,2}$};
		\draw [fill=black] (10,-1) circle (2.5pt);
		\draw[color=black] (10.1,-0.7) node {19};
		\draw [fill=black] (10,-1) circle (2.5pt);
		\draw[color=black] (10.2,-1.3) node {$c^0_{3,2}$};
		\draw [fill=black] (11,-1) circle (2.5pt);
		\draw[color=black] (11.1,-0.7) node {21};
		\draw [fill=black] (11,-1) circle (2.5pt);
		\draw[color=black] (11.2,-1.3) node {$c^0_{4,2}$};
		\draw [fill=black] (12,-1) circle (2.5pt);
		\draw[color=black] (12.1,-0.7) node {23};
		\draw [fill=black] (12,-1) circle (2.5pt);
		\draw[color=black] (12.2,-1.3) node {$c^0_{5,2}$};
		\draw [fill=black] (12,0) circle (2.5pt);
		\draw[color=black] (12.24,0.25) node {22};
		\draw [fill=black] (11,0) circle (2.5pt);
		\draw[color=black] (11.24,0.25) node {20};
		\draw [fill=black] (10,0) circle (2.5pt);
		\draw[color=black] (10.24,0.25) node {18};
		\draw [fill=black] (9,0) circle (2.5pt);
		\draw[color=black] (9.24,0.25) node {16};
		\draw [fill=black] (8,2) circle (2.5pt);
		\draw[color=black] (8.24,2.25) node {16};
		\draw [fill=black] (9,2) circle (2.5pt);
		\draw[color=black] (9.24,2.25) node {18};
		\draw [fill=black] (10,2) circle (2.5pt);
		\draw[color=black] (10.24,2.25) node {20};
		\draw [fill=black] (11,2) circle (2.5pt);
		\draw[color=black] (11.24,2.25) node {22};
		\draw [fill=black] (12,2) circle (2.5pt);
		\draw[color=black] (12.24,2.25) node {14};
		\draw [fill=black] (12,3) circle (2.5pt);
		\draw[color=black] (12.24,3.25) node {15};
		\draw [fill=black] (11,3) circle (2.5pt);
		\draw[color=black] (11.24,3.25) node {23};
		\draw [fill=black] (10,3) circle (2.5pt);
		\draw[color=black] (10.24,3.25) node {21};
		\draw [fill=black] (9,3) circle (2.5pt);
		\draw[color=black] (9.24,3.25) node {19};
		\draw [fill=black] (8,3) circle (2.5pt);
		\draw[color=black] (8.24,3.25) node {17};
		\draw [fill=black] (4,1) circle (2.5pt);
		\draw[color=black] (4,1.25) node {25};
		\draw [fill=black] (4,1) circle (2.5pt);
		\draw[color=black] (4,0.7) node {$u_{1,0}$};
		\draw [fill=black] (4,4) circle (2.5pt);
		\draw[color=black] (4,4.25) node {24};
		\draw [fill=black] (4,4) circle (2.5pt);
		\draw[color=black] (4,3.7) node {$u_{1,1}$};
		\draw [fill=black] (4,5.44) circle (2.5pt);
		\draw[color=black] (4,5.69) node {26};
		\draw [fill=black] (4,5.44) circle (2.5pt);
		\draw[color=black] (4,5.14) node {$u_{2,0}$};
		\end{scriptsize}
		\end{tikzpicture}
		\caption{\label{fig8} An $L(2,1)$-Labeling of $M^2(G)$ as in Theorem 4.6, where $\overline{G^2}$ has a perfect $2$-matching with two $K_2$ components and two cycles of order $3$ and $5$, here the edges represent a perfect matching of $\overline{G^2}\times K_2$.}		
	\end{figure}
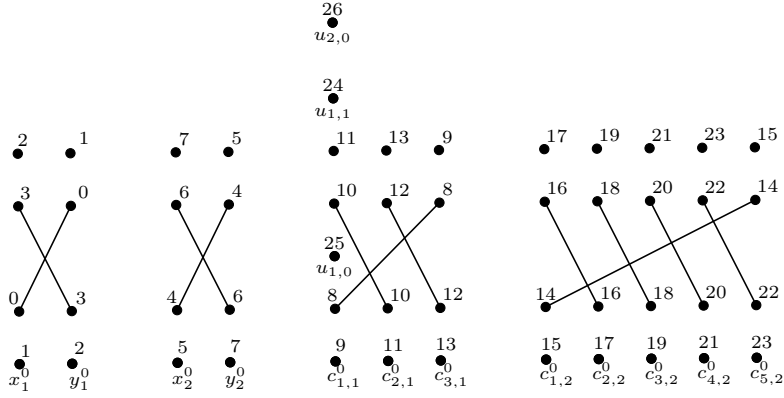
	 \par
	Since $d_{M^t}(u_{1,2^{t-1}-1},c^0_{n_s,s})=2$, $d_{M^t}(u_{1,2^{t-1}-1},c^{2^t-1}_{n_s-1,s})=2$, $d_{M^t}(u_{i,j},u_{i,j-1})=2$, and $d_{M^t}(u_{i,0},u_{i+1,2^{t-(i+1)}-1})=2$, this produces an $L(2,1)$-labeling with span $2^{t-1}(n+2)-2$. In Figure~\ref{fig8} an $L(2,1)$-labeling with the same schema for $M^2(G)$, where $\overline{G^2}$ has a perfect $2$-matching consisting of two $K_2$ components and  two cycles of order $3$ and $5$ respectively.	
	Hence from the lower bound of Theorem~\ref{th4.1} for $t\geq 2$, we have $\lambda(M^t(G))=2^{t-1}(n+2)-2$.
\end{proof}
The labeling defined in Theorem~\ref{th4.7} is a valid $L(2,1)$-labeling for any graph $G$ of order $n\geq 2$, if $\overline{G^2}$ has a perfect $2$-matching, then we can label the vertices of $M^t(G)$ with a labeling having span $2^{t-1}(n+2)-2$. Next, we give an upper bound for $\lambda(M^t(G))$ implying the maximum size of a $2$-matching of $\overline{G^2}$.
\begin{theorem}\label{th4.8}
	Let $G$ be a graph of order $n\geq 2$, with $\nu_2(\overline{G^2})=p$. For $t\geq 2$, we have $\lambda(M^t(G))\leq  2^{t-1}(2n-p+2)-2$.
\end{theorem}
\begin{proof}
	Let $G$ be a graph with $\nu_2(\overline{G^2})=p$. So there is an induced subgraph $H$ of $\overline{G^2}$ of order $p$, such that $H$ has a perfect $2$-matching. Let $V_H$ be the set of vertices of $H$, from Theorem~\ref{th4.7}, we can label the vertices of $M^t(G[V_H])$ with an $L(2,1)$-Labeling $f$ with span $2^{t-1}(p+2)-2$, where $f(u_{t,0})=2^{t-1}(p+2)-2$.\par 
	Now in $M^t(G)$, if $p<n$ the vertices remaining unlabeled by $f$ are the vertices in $V\setminus V_H$  and their copies. Let us denote $v^k_i$, where $1\leq i\leq q$, and $0\leq k\leq 2^t-1$, such that $p+q=n$, the vertices of $V\setminus V_H$ and their consecutive copies. Let $\chi_i$ with $2\leq i\leq q$, be a sequence of vertices in $M^t(G)$, where $\chi_i=v^2_iv^0_iv^1_i$ if $i$ is odd, and $\chi_i=v^1_iv^0_iv^2_i$ if $i$ is even. The only vertex labeled $2^{t-1}(p+2)-2$ by $f$ is $u_{t,0}$, using consecutive labels we label the vertices $v^k_i$, with $1\leq i\leq q$  beginning with the label $2^{t-1}(p+2)-1$, in the following order  $v^0_1v^2_1v^1_1\chi_2\ldots \chi_qv^3_qv^3_{q-1}\ldots v^3_1v^4_1\ldots v^4_qv^5_q\ldots v^{2^t-1}_1$.\par
	This produces an $L(2,1)$-labeling with span $2^{t-1}(p+2)-2+2^t(n-p)=2^{t-1}(2n-p+2)-2$.
\end{proof}
Similarly to Subsection~\ref{sec3.3}, we put interest in connected graphs, the path $P_n$ and cycle $C_n$, which we use to determine some connected graphs with the smallest $\lambda(M^t(G))$. 
\begin{cor}\label{cor4.3}
	For $t\geq 2$,	$$\lambda(M^t(P_n))=\begin{cases}
	4\times 2^t-2 \hspace{40pt}if\,\, n=3,4,5, \\
	2^{t-1}(n+2)-2 \hspace{15pt}if\,\,n \geq 6.	
	\end{cases}	$$
\end{cor} 
\begin{proof}	
	For $n=3$,  we have $diam(P_3)=2$, by Theorem~\ref{th4.4} for $t\geq 2$ we have $ \lambda(M^t(P_3))= 4\times 2^t-2$.\par	
	For $n=4$, $\overline{P^2_4}$ consists of a single edge and $2$ isolated vertices. So $\nu_2(\overline{P^2_4})=2$, it follows from Theorem~\ref{th4.8} that $\lambda(M^t(P_4))\leq 4\times 2^t-2$. Since $M^t(P_3)$ is a subgraph of $M^t(P_4)$, from above $\lambda(M^t(P_4))=4\times 2^t-2$.\par
	For $n=5$, $\overline{P^2_5}$ consists of $2$ independent edges and one isolated vertex. Hence $\nu_2(\overline{P^2_5})=4$, so from Theorem~\ref{th4.8} $\lambda(M^t(P_5))\leq 4\times 2^t-2$. Also $M^t(P_3)$ is a subgraph of $M^t(P_5)$, then $\lambda(M^t(P_5))=4\times 2^t-2$.\par
	For $n\geq 6$, it is easy to see that the path $P_n$ verifies the condition of Theorem~\ref{th4.7}, thus  $\lambda(M^t(P_n))=2^{t-1}(n+2)-2$.
\end{proof}
\begin{cor}\label{cor4.4}
	For $t\geq 2$,	$$\lambda(M^t(C_n))=\begin{cases}
	4\times 2^t-2 \hspace{40pt}if\,\, n=3, \\
	5\times 2^t-2 \hspace{40pt}if\,\, n=4, \\
	6\times 2^t-2 \hspace{40pt}if\,\, n=5, \\
	2^{t-1}(n+2)-2 \hspace{15pt}if\,\,n \geq 6.	
	\end{cases}$$	
\end{cor} 
\begin{proof}
	We have $diam(C_3)=1$, and  $diam(C_4)=diam(C_5)=2$. So by Theorem~\ref{th4.4}, for $t\geq 2$, we have $\lambda(M^t(C_3))=4\times 2^t-2$, $\lambda(M^t(C_4))=5\times 2^t-2$, and $\lambda(M^t(C_5))=6\times 2^t-2$. If $n\geq 6$, the cycle $C_n$ satisfies the condition of Theorem~\ref{th4.7}, then $\lambda(M^t(C_n))=2^{t-1}(n+2)-2$.
\end{proof}
\begin{cor}\label{cor4.5}
	Let $G$ be a connected graph, for $t\geq 2$ we have \\
	$1)$	$\lambda(M^t(G))=3\times 2^t-2$   if and only if  $G$ is $K_2$,\\
	$2)$	$\lambda(M^t(G))=4\times 2^t-2$   if and only if  $G \in \{ P_3, P_4, P_5,P_6,C_3,C_6  \}$,\\
	$3)$	$\lambda(M^t(G))=9\times 2^{t-1}-2$   if and only if  $G \in \{ P_7,C_7 \}$.		
\end{cor} 
\begin{proof}
	We have $K_2$ is the only graph with $\bigtriangleup=1$, by Theorem~\ref{th4.2} $\lambda(M^t(K_2))=3\times 2^t-2$. From the lower bound of Theorem~\ref{th4.1}, for $t\geq2$, $\lambda(M^t (G)) \geq 2^{t-1}max(n+2, 2(\bigtriangleup+2))-2$.  Then if $\bigtriangleup\geq 2$ we have $\lambda(M^t (G))\geq 4\times 2^t-2$. Also if $\bigtriangleup\geq 3$, we have $\lambda(M^t (G))\geq 5\times 2^t-2$. Since the graphs in  Corollary~\ref{cor4.3} and Corollary~\ref{cor4.4} are the only connected graphs with $\bigtriangleup=2$, then we can conclude. 
\end{proof}
For any other non-trivial connected graph $G$ not mentioned in Corollary~\ref{cor4.5} for $t\geq 2$, we have $\lambda(M^t(G))\geq 5\times 2^t-2$. 
\section{Open problems}
From the statement of the $\bigtriangleup^2$-conjecture, and the upper bound of Theorem~\ref{th3.1} and Theorem~\ref{th4.1}, we propose a weaker conjecture for the $L(2,1)$-labeling number of the Mycielski and the iterated Mycielski of graphs.
\begin{con}
	For any graph $G$ of order $n\geq 1$, with maximum degree $\bigtriangleup$, and for all $t\geq 1$, we have $\lambda(M^t (G))\leq  (2^t-1)(n+1)+\bigtriangleup^2$.
\end{con}
It is clear from Theorem~\ref{th3.1} and Theorem~\ref{th4.1} that if $\lambda(G)\leq \bigtriangleup^2$, then for any $t\geq 1$, $\lambda(M^t (G))\leq  (2^t-1)(n+1)+\bigtriangleup^2$, also if it is true for an iteration $t$ then it is for any iteration greater. From our study, for any $t\geq 1$, the only graphs with at least one edge that we know having $\lambda(M^t (G))=(2^t-1)(n+1)+\bigtriangleup^2$, are the graph $K_2$, and the graphs achieving the bound in Corollary~\ref{cor3.1}, which are the cycle $C_5$, the Petersen graph, the Hoffman-Singleton graph, and possibly a diameter two Moore graph of maximum degree $57$, and order $57^2+1$ if such graph exists.  
\par The complexity of the $L(2,1)$-labeling problem for the Mycielski of graphs should be more investigated, whether for general graphs or the Mycielski of graphs not still studied. For instance, trees since the $L(2,1)$-labeling number can be determined in polynomial time for trees \cite{changkuo}, we may ask if it is also the case for the Mycielski graphs generated from trees?


\begin{thebibliography}{99}
\bibitem{kano} A. Amahashi, M. Kano, \textit{On factors with given components}, Discrete Mathematics \textbf{42}  (1982) 1-6.\\ https://doi.org/10.1016/0012-365X(82)90048-6
\bibitem{boro} M.L. Borowiecki, P. Borowiecki, E. Drgas-Burchardt, E. Sidorowicz, \textit{Graph Classes Generated by Mycielskians}, Discussiones Mathematicae Graph Theory \textbf{40}  (2020) 1163-1173.\\ https://doi.org/10.7151/dmgt.2345
\bibitem{tiziana} T. Calamoneri, \textit{The L(h, k)-Labelling Problem: An Updated Survey and Annotated Bibliography},  The Computer Journal  \textbf{54} (2011) 1344-1371. \\ https://doi.org/10.1093/comjnl/bxr037
\bibitem{cara}	M. Caramia, P. Dell'Olmo, \textit{A lower bound on the chromatic number of Mycielski graphs},  Discrete Mathematics  \textbf{235} (2001) 79-86. \\ https://doi.org/10.1016/S0012-365X(00)00261-2
\bibitem{changhuang}	G.J. Chang, L. Huang, X. Zhu, \textit{Circular chromatic numbers of Mycielski's graphs}, Discrete Mathematics  \textbf{205} (1999) 23-37. \\ https://doi.org/10.1016/S0012-365X(99)00033-3
\bibitem{changkuo}	G.J. Chang, D. Kuo, \textit{The $L(2,1)$-Labeling Problem on Graphs}, SIAM J. Discrete Math.  \textbf{9} (1996) 309-316. \\ https://doi.org/10.1137/S0895480193245339	
\bibitem{fiala} J. Fiala, T. Kloks, J. Kratochv\'{\i}l, \textit{Fixed-parameter complexity of $\lambda$-labelings}, Discrete Applied Mathematics \textbf{113} (2001) 59-72. \\ https://doi.org/10.1016/S0166-218X(00)00387-5
\bibitem{fisher} D.C. Fisher, P.A. McKenna, E.D. Boyer, \textit{Hamiltonicity, diameter, domination, packing, and biclique partitions of Mycielski's graphs}, Discrete Applied Mathematics \textbf{84} (1998) 93-105. \\ https://doi.org/10.1016/S0166-218X(97)00126-1
\bibitem{george} J.P. Georges, D.W. Mauro, M.A. Whittlesey,  \textit{Relating path coverings to vertex labellings with a condition at distance two}, Discrete Mathematics  \textbf{135} (1994) 103-111. \\ https://doi.org/10.1016/0012-365X(93)E0098-O
\bibitem{goncal} D. Gonçalves,  \textit{On the L(p,1)-labelling of graphs},  Discrete Mathematics \textbf{308} (2008) 1405-1414. \\ https://doi.org/10.1016/j.disc.2007.07.075	
\bibitem{griggs} J.R. Griggs, R.K. Yeh,  \textit{Labelling Graphs with a Condition at Distance 2}, SIAM J. Discrete Math.  \textbf{5} (1992) 586-595. \\ https://doi.org/10.1137/0405048
\bibitem{hale}	W.K. Hale, \textit{Frequency Assignment: Theory and Applications}, Proc. of the IEEE  \textbf{68} (1981) 1497-1514. \\ https://doi.org/10.1109/PROC.1980.11899	
\bibitem{havet}	F. Havet, B. Reed, J.-S. Sereni \textit{L(2,1)-labelling of Graphs}, in: Proceedings of the Nineteenth Annual ACM-SIAM Symposium on Discrete Algorithms (SODA08), San Francisco, California, 20-22 January, 2008, pp. 621–630.
\bibitem{hoffman} A.J. Hoffman, R.R. Singleton, \textit{On Moore Graphs with Diameters 2 and 3}, IBM Journal of Research and Development \textbf{4} (1960) 497-504. \\ https://doi.org/10.1147/rd.45.0497
\bibitem{kirk}	D.G. Kirkpatrick, P. Hell, \textit{On the Complexity of General Graph Factor Problems},  SIAM Journal on Computing \textbf{12} (1983) 601-609. \\ https://doi.org/10.1137/0212040		
\bibitem{larsen} M. Larsen, J. Propp, D. Ullman, \textit{The fractional chromatic number of mycielski's graphs}, Journal of Graph Theory  \textbf{19} (1995) 411-416. \\ https://doi.org/10.1002/jgt.3190190313
\bibitem{lin}	W. Lin, P. Lam, \textit{Star matching and distance two labelling}, Taiwanese Journal of Mathematics  \textbf{13} (2009) 211-224. \\ https://doi.org/10.11650/tjm.13.2009.539
\bibitem{lovas}	L. Lovász, M.D. Plummer, Matching Theory, in: Annal of Discrete Mathematics, vol. 29, North-Holland, Amsterdam, 1986.	
\bibitem{mycielski} J. Mycielski,  \textit{Sur le coloriage des graphs}, Colloquium Mathematicae  \textbf{2} (1955) 161-62.	
\bibitem{shao}	Z. Shao, R. Solis-Oba, \textit{Labeling Mycielski Graphs with a Condition at Distance Two}, Ars Comb. \textbf{140} (2018) 337-349.	
\bibitem{tutte}	W.T. Tutte, \textit{The $1$-factors of oriented graphs},  Proc. Amer. Math. Soc. \textbf{4} (1953) 922-931. \\ https://doi.org/10.1090/S0002-9939-1953-0063009-7
\bibitem{las} M.L. Vergnas, \textit{An extension of Tutte's 1-factor theorem}, Discrete Mathematics \textbf{23}  (1978) 241-255. \\ https://doi.org/10.1016/0012-365X(78)90006-7
\bibitem{dougla}  D.B. West, Introduction to Graph Theory, second ed., Prentice-Hall, Englewood cliffs, NJ, 2001.
\bibitem{yeh}	R.K. Yeh, \textit{A survey on labeling graphs with a condition at distance two},  Discrete Mathematics \textbf{306} (2006) 1217-1231. \\ https://doi.org/10.1016/j.disc.2005.11.029
\end{thebibliography}
\end{document}